\documentclass[journal]{aiaa-pretty}    % journal, submit, article
\usepackage{amssymb}
\usepackage{amsmath}
\usepackage{booktabs}
 \usepackage{varioref}%  smart page, figure, table, and equation referencing
 \usepackage{wrapfig}%   wrap figures/tables in text (i.e., Di Vinci style)
 \usepackage{threeparttable}% tables with footnotes
 \usepackage{dcolumn}%   decimal-aligned tabular math columns
  \newcolumntype{d}{D{.}{.}{-1}}
 \usepackage{nomencl}%   nomenclature generation via makeindex
  \makeglossary
 \usepackage{subfigure}% subcaptions for subfigures
 \usepackage{graphicx}
 \usepackage{subfigmat}% matrices of similar subfigures, aka small mulitples
 \usepackage{fancyvrb}%  extended verbatim environments
  \fvset{fontsize=\footnotesize,xleftmargin=2em}
 \usepackage{lettrine}%  dropped capital letter at beginning of paragraph
\usepackage{enumerate}
\newtheorem{theorem}{Theorem}

\newtheorem{lemma}{Lemma}
\newtheorem{remark}{Remark}
\newtheorem{definition}{Definition}

\newenvironment{proof}{{\it Proof. }}{\hfill $\Box$}

\newcommand{\Real}{\mathbb R}
\newcommand{\set}[1]{\left\{#1\right\}}
\newcommand{\real}[1]{{\mathbb R}^{#1}}
\newcommand{\be}{{\boldsymbol e}}
\newcommand{\bff}{{\boldsymbol f}}

\newcommand{\bh}{{\boldsymbol h}}

\newcommand{\bp}{{\boldsymbol p}}
\newcommand{\bq}{{\boldsymbol q}}

\newcommand{\bu}{{\boldsymbol u}}

\newcommand{\bv}{{\boldsymbol v}}

\newcommand{\bx}{\boldsymbol x}

\newcommand{\bxf}{{\bx(\cdot)}}  % f for function
\newcommand{\buf}{{\bu(\cdot)}}  % f for function
\newcommand{\bA}{{\boldsymbol A}}

\newcommand{\bD}{{\boldsymbol D}}
\newcommand{\bG}{{\boldsymbol G}}

\newcommand{\bK}{{\boldsymbol K}}

\newcommand{\bQ}{{\boldsymbol Q}}

\newcommand{\X}{\mathbb{X}}
\newcommand{\U}{\mathbb{U}}

\newcommand{\bzero}{{\bf 0}}

\newcommand{\bnu}{{\mbox{\boldmath $\nu$}}}
\newcommand{\bmu}{{\mbox{\boldmath $\mu$}}}

\newcommand{\blam}{{\mbox{\boldmath $\lambda$}}}

\newcommand{\bomega}{\mbox{\boldmath$\omega$}}

\author{ %
I. M. Ross\thanks{Author of the textbook, \textit{A Primer on Pontryagin's Principle on Optimal Control}, Collegiate Publishers, San Francisco, CA, 2015. } \\
\textit{Monterey, CA 93940}
}
\title{Enhancements to the DIDO$^\copyright$ Optimal Control Toolbox}

\abstract{
In 2020, DIDO$^\copyright$ turned 20! The software package emerged in 2001 as a basic, user-friendly MATLAB$^\circledR$ teaching-tool to illustrate the various nuances of Pontryagin's Principle but quickly rose to prominence in 2007 after NASA announced it had executed a globally optimal maneuver using DIDO.  Since then, the toolbox has grown in applications well beyond its aerospace roots: from solving problems in quantum control to ushering rapid, nonlinear sensitivity-analysis in designing high-performance automobiles.  Most recently, it has been used to solve continuous-time traveling-salesman problems. Over the last two decades, DIDO's algorithms have evolved from their simple use of generic nonlinear programming solvers to a multifaceted engagement of fast spectral Hamiltonian programming techniques. A description of the internal enhancements to DIDO that define its mathematics and algorithms are described in this paper. A challenge example problem from robotics is included to showcase how the latest version of DIDO is capable of escaping the trappings of a ``local minimum'' that ensnare many other trajectory optimization methods.

}

\begin{document}
\maketitle

\section{Introduction}\label{sec:intro}

On November 5, 2006, the International Space Station executed a globally optimal maneuver that saved NASA \$1,000,000\cite{zpm:NASA-report}.  This historic event added one more first\cite{zpm:first} to many of NASA's great accomplishments. It was also a first flight demonstration of DIDO$^\copyright$, which was then a basic optimal control solver.  Because DIDO was based on a then new pseudospectral (PS) optimal control theory, the concept emerged as a big winner in computational mathematics as headlined on page~1 of \textit{SIAM News}\cite{SIAMpg1}. Subsequently, \textit{IEEE Control Systems Magazine} did a cover story on the use of PS control theory for attitude guidance\cite{zpm:IEEE}. In exploiting the widespread publicity in the technical media (see Fig.~\ref{fig:ZPMcollage}) ``new'' codes based on naive PS discretizations materialized with varying degrees of disingenuous claims. Meanwhile, the algorithm implemented in DIDO moved well beyond its simple use of generic nonlinear programming solvers\cite{lncis,knots} to a multifaceted guess-free, spectral algorithm\cite{guess-free,spec-alg}. That is, while the 2001-version of DIDO was indeed a simple code that patched PS discretization to a generic nonlinear programming solver, it quietly and gradually evolved to an entirely new approach that we simply refer to as  \emph{\textbf{DIDO's algorithm}}.  DIDO's algorithm is actually a suite of  algorithms that work in unison and focused on three key performance elements:
\begin{enumerate}
\item Extreme robustness to the point that a ``guess'' is not required from a user\cite{guess-free};
\item Spectral acceleration based on the principles of discrete cotangent tunneling (see Theorem~\ref{theorem:tunnel} in Appendix A of this paper); and
\item Verifiable accuracy of the computed solution in terms of feasibility and optimality tests\cite{ross-book}.
\end{enumerate}
Despite its current internal changes, the usage of DIDO from a user's perspective has remained virtually the same since 2001. This is because DIDO has remained true to its founding principle: to provide a minimalist's approach for solving optimal control problems. In showcasing this philosophy, we first describe DIDO's format of an optimal control problem.  This format closely resembles writing the problem by hand (``pen-to-paper''). The input to DIDO is just the problem definition. The outputs of DIDO include the collection of necessary conditions generated by applying Pontryagin's Principle\cite{ross-book}.  Because the necessary conditions are in the \emph{\textbf{outputs}} of DIDO (and not its inputs!) a user may apply Pontryagin's Principle to ``independently'' test the optimality of the computed solution.

The input-output features of DIDO are in accordance with Pontryagin's Principle whose fundamentals have largely remained unchanged since its inception, circa 1960\cite{pontryagin}. Appendixes~A and B describe the mathematical and algorithmic details of DIDO; however, no knowledge of its inner workings is necessary to employ the toolbox effectively.  The only apparatus that is needed to use DIDO to its full capacity is a clear understanding of Pontryagin's Principle.  DIDO is based on the philosophy that a good software must solve a problem as efficiently as possible with the least amount of iterative hand-tuning by the user.   This is because it is more important for a user to focus on solving the ``problem-of-problems'' rather than expend wasteful energy on the minutia of coding and dialing knobs\cite{ross-book}.  These points are amplified in Sec.~\ref{sec:robotics} of this paper by way of a challenge problem from robotics to show how DIDO can escape the trappings of a ``local minimum'' and emerge on the other side to present a viable solution.

\section{DIDO's Format and Structures for a Generic Problem}\label{sec:DIDO-i/o}
A generic two-event\footnote{The two-event problem may be viewed as an ``elementary'' continuous-time traveling salesman problem\cite{TSP20, ACC19, nolcos16}. } optimal control problem may be formulated as\footnote{Although the notation used here is fairly self-explanatory, see \cite{ross-book} for additional clarity.}
%\footnote{The notation used here is ``modern'' and quite universal; see \cite{ross-book} for details. }
%
%==========================================================
\newlength{\leftside}
\newlength{\rightside}
\newcommand*{\leftterm}{}
\newcommand*{\rightterm}{}
\newcommand*{\term}[1]{$\displaystyle#1$}
%============================================================================
\begin{align*}
%---------------------------------------------------
&\left.\begin{aligned}
\phantom{problem}
\X & \subset \real{N_x}   &\U & \subset \real{N_u}& \\
\bx &= (x_1, \ldots, x_{N_x})   &\bu &= (u_1, \ldots, u_{N_u}) & \\
\end{aligned}\hspace{0.01cm} \right\} \ \text{\emph{\textbf{preamble}}}
%---------------------------------------------------
\\[0.5em]
%\]
%\vspace{-18pt} % USE THIS TO ALIGN!!
%%==========================================================
%\newlength{\leftside}
%\newlength{\rightside}
%\newcommand*{\leftterm}{}
%\newcommand*{\rightterm}{}
%\newcommand*{\term}[1]{$\displaystyle#1$}
%-----------------------------------------------------------
%\[
%
&\begin{aligned}
\renewcommand*{\leftterm}{\phantom{J[\bx(\cdot), \bu(\cdot), t_0, t_f]}}
\renewcommand*{\rightterm}{\phantom{(\bx_0, x t_f) }}
\settowidth{\leftside}{\term{\leftterm}}
\settowidth{\rightside}{\term{\rightterm}}
%----------------------------------------------------------
%%---------------------------------------------------
%\begin{array}{rlrlrl}
%\X &=\real{3}   &\U &= \Real& \\
%\bx &= (x, y, v)  \quad &\bu &= \theta & \\
%\end{array}\\
%%---------------------------------------------------
\overbrace{(\bG)}^{\text{\normalsize \emph{\textbf{problem}}}} \left\{
\begin{array}{ll}
\text{Minimize}& \\[-1.20em] % lines up Minimize with J
& \left.\begin{aligned}
\makebox[\leftside][l]{\term{\phantom{J[\bx(\cdot), \bu(\cdot), t_0,]} }} & \\[0.5em]
\makebox[\leftside][r]{\term{J[\bx(\cdot), \bu(\cdot), t_0, t_f, \bp] := }} & \\[0.5em]
% \makebox[\leftside][r]{\term{= E(\bx_0, \bx_f, t_0, t_f)}} & \\ %+ \makebox[\rightside][l]{\term{Fsf}}\\
 \makebox[\leftside][r]{\term{E(\bx_0, \bx_f, t_0, t_f, \bp)}} &  \\[0.5em]
 % \makebox[\rightside][l]{\term{\int_{t_0}^{t_f}F(\bx(t), \bu(t), t)\ dt }}  \\
  \makebox[\leftside][l]{\term{+ \int_{t_0}^{t_f}F(\bx(t), \bu(t), t, \bp)\ dt }} & \makebox[\rightside][l]{\term{\phantom{a}}}\\
 \end{aligned}\right\} \ \text{\textbf{\emph{cost}}}\\ [3.7em]
\text{Subject to} &  %\begin{aligned}
%    \makebox[\leftside][r]{\term{\dot x}} &= \makebox[\rightside][l]{\term{v\sin\theta}}\\
%  \end{aligned}
\\[-1.2em]
&  \left.\begin{aligned}
%    \makebox[\leftside][r]{\term{\dot\bx}} &= \\
    \makebox[\leftside][l]{\term{\dot\bx = \bff(\bx(t), \bu(t), t, \bp)}} & \makebox[\rightside][l]{\term{\phantom{a}}} \\
  \end{aligned}\right\}\ \text{\emph{\textbf{dynamics}}} \\ [1em]
  %
  %%
%&  \left.\begin{aligned}
%    \makebox[\leftside][r]{\term{\dot x}} &= \makebox[\rightside][l]{\term{v\sin\theta}} \\
%    \dot y &= v\cos\theta \\
%    \dot v &= g\cos\theta
%  \end{aligned}\right\} \quad \textsf{(dynamics)} \\[2em]
%  %
&  \left.\begin{aligned}
   % \makebox[\leftside][r]{\term{\be^L}}&\le \\
    \makebox[\leftside][r]{\term{\be^L \le \be(\bx_0, \bx_f, t_0, t_f, \bp) \le \be^U}}
   &    \makebox[\rightside][l]{\term{\phantom{a}}}
    \\
%    \makebox[\leftside][r]{\term{\be^L \le \be(\bx_0, \bx_f, t_0, t_f)}}&\le \makebox[\rightside][l]{\term{\be^U}}\\
%    \be^L &=  (0, 0)
  \end{aligned}\right\} \ \text{\emph{\textbf{events}}} \\[1em]
  &  \left.\begin{aligned}
%    (t_0, x_0, y_0, v_0) &=  \makebox[\rightside][l]{\term{(0, 0, 0, 0)}}\\
%   \makebox[\leftside][r]{\term{\bh^L}}
%   &\le \makebox[\rightside][l]{\term{\bh(\bx(t), \bu(t), t) \le \bh^U}}\\
   \makebox[\leftside][r]{\term{\bh^L \le \bh(\bx(t), \bu(t), t, \bp) \le \bh^U}}
   & \makebox[\rightside][l]{\term{\phantom{a}}}\\
  \end{aligned}\right\} \ \text{\emph{\textbf{path}}} \\[1em]
%----------------------------------------
    &  \left.\begin{aligned}
   \makebox[\leftside][r]{\term{\bK^L \le \bK[\bx(\cdot), \bu(\cdot), t_0, t_f, \bp]\ }}
   & \makebox[\rightside][l]{\term{ \le \bK^U}}\\
  \end{aligned}\right\} \ \text{\emph{\textbf{functional}}}
\end{array}
\right.
\end{aligned}
\end{align*}
%==================================================================
%
The basic continuous-time optimization variables are packed in the DIDO structure called \textsf{primal} according to:
\begin{subequations}\label{eq:primal_def_1of2}
\begin{align}
&\textsf{primal}.\textbf{states}    & \longleftarrow  &&&    \bx(\cdot)\\
&\textsf{primal}.\textbf{controls}  & \longleftarrow  &&&    \bu(\cdot)\\
&\textsf{primal}.\textbf{time}      & \longleftarrow   &&&    t
%&\textsf{primal}.\textbf{parameters}& \longleftarrow  &&&    \bp
\end{align}
\end{subequations}
Additional optimization variables packed in \textsf{primal} are:
\begin{subequations}\label{eq:primal_def_2of2}
\begin{align}
&\textsf{primal}.\textbf{initial.states}    & \longleftarrow  &&&    \bx_0\\
&\textsf{primal}.\textbf{initial.time}  & \longleftarrow  &&&    t_0\\
&\textsf{primal}.\textbf{final.states}      & \longleftarrow   &&&    \bx_f \\
&\textsf{primal}.\textbf{final.time}      &\longleftarrow &&& t_f\\
&\textsf{primal}.\textbf{parameters}& \longleftarrow  &&&    \bp
\end{align}
\end{subequations}
The names of the five functions that define Problem~$(\bG)$ are stipulated in the structure \textsf{problem} according to:
\begin{subequations}\label{eq:problem.functions=def}
\begin{align}
&\textsf{problem}.\textbf{cost}    & \leftarrow  &&& [E(\bx_0, \bx_f, t_0, t_f, \bp ), \nonumber\\
&                                   &                 &&& \quad F(\bx(t), \bu(t), t, \bp)]\\
&\textsf{problem}.\textbf{dynamics}  & \leftarrow  &&&    \bff(\bx(t), \bu(t), t, \bp)\\
&\textsf{problem}.\textbf{events}      & \leftarrow   &&&    \be(\bx_0, \bx_f, t_0, t_f, \bp) \\
&\textsf{problem}.\textbf{path}      &\leftarrow &&& \bh(\bx(t), \bu(t), t, \bp) & \text(optional)\\
&\textsf{problem}.\textbf{functionals}& \leftarrow  &&&    \bK[\bx(\cdot), \bu(\cdot), t_0, t_f, \bp] &\text(optional)
\end{align}
\end{subequations}
%-------------------------------------------------------------
The bounds on the constraint functions of Problem~$(\bG)$ are also associated with the structure \textsf{problem} but under its fieldname of \textbf{bounds} (i.e. \textsf{problem}.\textbf{bounds}) given by\footnote{Because some DIDO clones still employ the old DIDO format (e.g.,  ``bounds.lower'' and ``bounds.upper''), the user-specific codes must be remapped to run on such imitation software. },
\begin{subequations}
\begin{align}
&\textsf{bounds}.\textbf{events}      & \leftarrow   &&&    [\be^L, \be^U] \\
&\textsf{bounds}.\textbf{path}      &\leftarrow &&& [\bh^L, \bh^U]  & \text(optional)\\
&\textsf{bounds}.\textbf{functionals}& \leftarrow  &&&   [\bK^L, \bK^U] &\text(optional)
\end{align}
\end{subequations}
%-------------------------------------------------------------
Although it is technically unnecessary, a \emph{\textbf{preamble}} is highly recommended and included in all the example problems that come with the full toolbox.\footnote{DIDOLite, the free version of DIDO, has limited capability.}  The preamble is just a simple way to organize all the primal variables in terms of symbols and notation that have a useful meaning to the user.  In addition, the field \textbf{constants} is included in both \textsf{primal} and \textsf{problem} so that problem-specific constants can be managed with ease. Thus, \textsf{problem}.\textbf{constants} is where a user defines the data that can be accessed anywhere in the user-specific files by way of \textsf{primal}.\textbf{constants}.

A ``search space'' must be provided to DIDO using the field \textbf{search} under \textsf{problem}. The search space for states and controls are specified by
\begin{subequations}\label{eq:search-space}
\begin{align}
&\textsf{search}.\textbf{states}      & \leftarrow   &&&    [\bx^{LL},\quad \bx^{UU}] \\
&\textsf{search}.\textbf{controls}      &\leftarrow &&& [\bu^{LL}, \quad \bu^{UU}]
\end{align}
\end{subequations}
%-------------------------------------------------------------
where, $(\bx, \bu)^{LL}$ and $(\bx, \bu)^{UU}$ are nonactive bounds on the state and control variables; hence, \textbf{\emph{they must be nonbinding constraints}}.

The inputs to DIDO are simply two structures given by \textsf{problem} and \textsf{algorithm}.
%%
%$$\text{\textsf{\textbf{dido}}\big(\textsf{problem}, \textsf{algorithm}\big)}  $$
%%
The \textsf{algorithm} input provides some minor ways to test certain algorithmic options.  The reason the algorithm structure is not extensive is precisely this: To maintain the DIDO philosophy of %allowing the user to focus on solving the problem rather than
not wasting time in fiddling with various knobs that only serves to distract a user from solving the problem.  That is, it is not good-enough for DIDO to run efficiently; it must also provide a user a quick and easy approach in taking a problem from \emph{\textbf{concept-to-code}}\footnote{Warning: The quick and easy concept-to-code philosophy does not mean a user should code a problem first and ask questions later.  DIDO is user-friendly, but it's not for dummies! Nothing beats analysis before coding; see Sec.~4.1.4 of \cite{ross-book} for a simple ``counter'' example and Sec.~\ref{sec:robotics} of this paper for the recommended process.}.

Before taking a concept to code fast, it is critically important to perform a mathematical analysis of the problem; i.e., an investigation of the necessary conditions.  This investigation is necessary (pun intended!) not only to understand DIDO's outputs but also to eliminate common conceptual errors that can easily be detected by a quick rote application of Pontryagin's Principle.
%
%======================================================================================
   \begin{figure}[h!]
      \centering
      \framebox{\parbox{0.9\columnwidth}{
      \centering
      {\includegraphics[angle=0, width = 0.9\columnwidth]{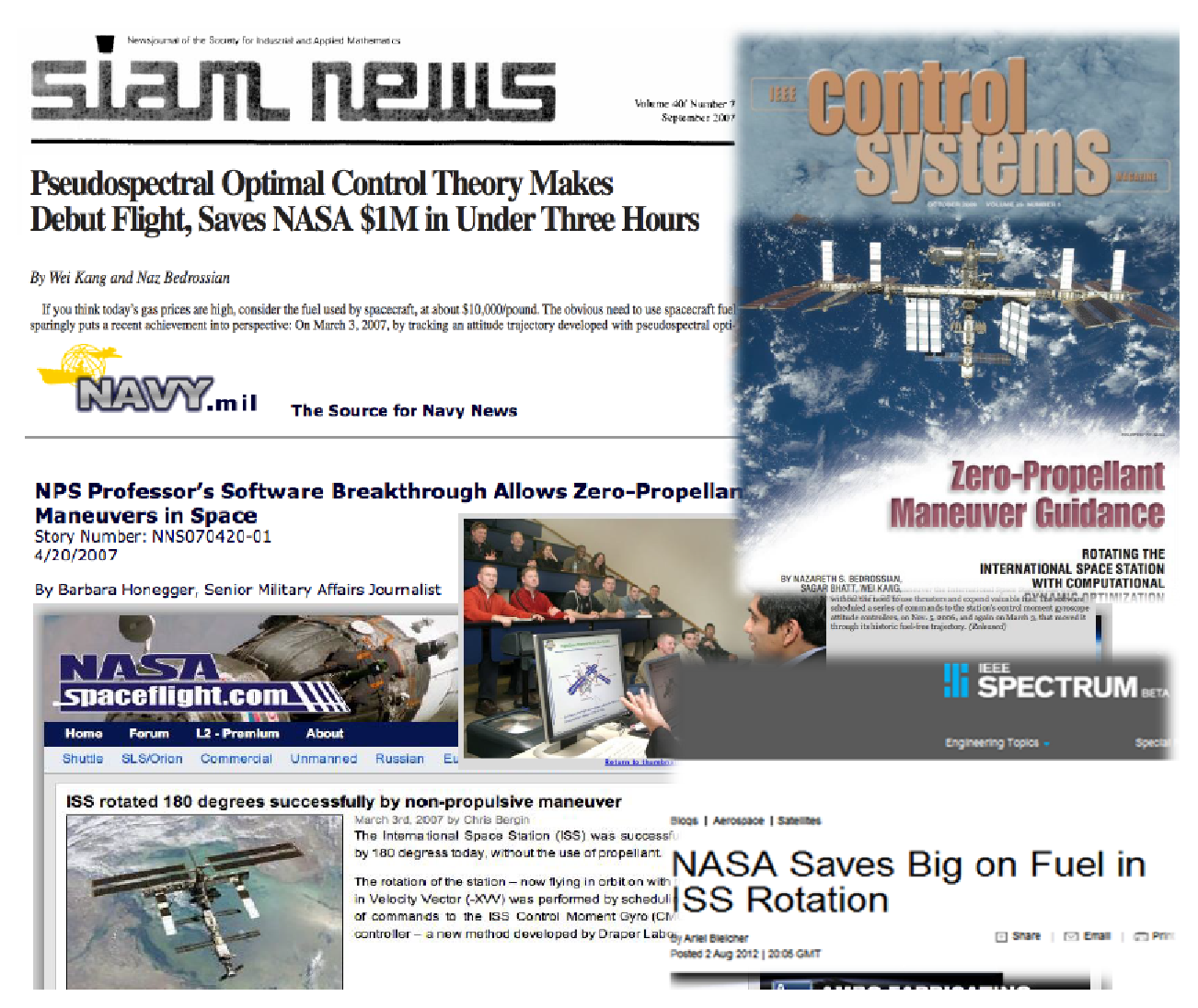}}
      \caption{\textsf{Sample media coverage of NASA's successful flight operations using PS optimal control theory as implemented in DIDO.}}
    \label{fig:ZPMcollage}
    }}
   \end{figure}
%==========================================================================================
%

Although proving Pontryagin's Principle is quite hard, its application to any problem is quite easy!  See \cite{ross-book} for details on the ``HAMVET'' process.  The first step in this process is to construct the Pontryagin Hamiltonian. For Problem~$(\bG)$ this is given by\footnote{We ignore the functional constraints ($\bK[\cdot]$) in the rest of this paper for simplicity of presentation.},
\begin{equation}\label{eq:H=bydef}
H(\blam, \bx, \bu, t, \bp) := F(\bx, \bu, t, \bp) + \blam^T\bff(\bx, \bu, t, \bp)
\end{equation}
That is, constructing \eqref{eq:H=bydef} is as simple as taking a dot product. Forming the \emph{\textbf{Lagrangian of the Hamiltonian}} requires taking another dot product given by,
\begin{equation}\label{eq:Hbar=bydef}
\overline{H}(\bmu, \blam, \bx, \bu, t, \bp) := H(\blam, \bx, \bu, t, \bp) + \bmu^T\bh(\bx, \bu, t, \bp)
\end{equation}
where, $\bmu$ is a path covector (multiplier) that satisfies the \emph{\textbf{complementarity conditions}}, denoted by $\bmu\dagger\bh $, and given by,
\begin{equation}\label{eq:comp-h}
\bmu\dagger \bh \Leftrightarrow \mu_{i} \ \left\{\begin{array}{ccc}
  \le 0 & \ \text{if} \ \  & h_i(\bx, \bu, t, \bp) = h_i^L \\
  = 0 & \ \text{if} \ \ & h_i^L < h_i(\bx, \bu, t, \bp) < h_i^U \\
  \ge 0 & \ \text{if} \ \   & h_i(\bx, \bu, t, \bp) = h_i^U \\
  unrestricted & \ \text{if} \ \   & h_i^L = h_i^U
\end{array}  \right.
\end{equation}
Constructing the \emph{\textbf{adjoint equation}} now requires some calculus of generating gradients:
\begin{equation}
-\dot\blam = \frac{\partial \overline{H}}{\partial \bx}
\end{equation}
The adjoint covector (costate) satisfies the \emph{\textbf{tranversality conditions}} given by,
\begin{align}
- \blam(t_0) &= \frac{\partial \overline{E}}{\partial \bx_0}
&\blam(t_f) &= \frac{\partial \overline{E}}{\partial \bx_f}
\end{align}
where, $\overline{E}$ is the \emph{\textbf{Endpoint Lagrangian}} constructed by taking yet another dot product given by,
\begin{multline}\label{eq:Ebar=bydef}
\overline{E}(\bnu, \bx_0, \bx_f, t_0, t_f, \bp) := E(\bx_0, \bx_f, t_0, t_f, \bp) \\+ \bnu^T\be(\bx_0, \bx_f, t_0, t_f, \bp)
\end{multline}
and $\bnu$ is the endpoint covector that satisfies the complementarity condition $\bnu\dagger\be$ (see \eqref{eq:comp-h}). The values of the  Hamiltonian at the endpoints also satisfy transversality conditions known as the \emph{\textbf{Hamiltonian value conditions}} given by,
\begin{align}
\mathcal{H}[@t_0] &= \frac{\partial \overline{E}}{\partial t_0}
&-\mathcal{H}[@t_f] &= \frac{\partial \overline{E}}{\partial t_f} &(\text{and } \bnu\dagger\be )
\end{align}
where, $\mathcal{H}$ is the \emph{\textbf{lower Hamiltonian}}\cite{ross-book,clarke-2013book}. The \emph{\textbf{Hamiltonian minimization condition}}, which is frequently and famously stated as
\begin{equation}\label{eq:HMC}
\begin{array}{cc}
  (\text{HMC})& \left\{\begin{array}{ll}
    \displaystyle\mathop\text{Minimize }_\bu & H(\blam, \bx, \bu, t, \bp) \\
    \text{Subject to} & \bu \in \mathbb{U} \
  \end{array}\right.
\end{array}
\end{equation}
reduces to the function-parameterized nonlinear optimization problem given by,
\begin{equation}\label{eq:HMC=NLP}
\begin{array}{cc}
  (\text{HMC for (\textbf{G})})& \left\{\begin{array}{ll}
    \displaystyle\mathop\text{Minimize }_\bu & H(\blam, \bx, \bu, t, \bp) \\
    \text{Subject to} & \bh^{L} \leq \bh(\bx, \bu, t, \bp)\leq \bh^{U}
  \end{array}\right.
\end{array}
\end{equation}
The necessary conditions for the subproblem stated in \eqref{eq:HMC=NLP} are given by the Karush-Kuhn-Tucker (KKT) conditions:
\begin{align}
\frac{\partial \overline{H}}{\partial \bu} &= \bzero   & \bmu \dagger \bh
\end{align}
Similarly, the necessary condition for selecting an optimal parameter $\bp$ is given by,
\begin{equation}\label{eq:HMCfix4p}
\frac{\partial \overline{E}}{\partial \bp} + \displaystyle \int_{t_0}^{t_f} \left(\frac{\partial \overline{H}}{\partial \bp}\right)\, dt = \bzero
\end{equation}
Finally, the \emph{\textbf{Hamiltonian evolution equation}},
\begin{equation}
\frac{d \mathcal{H}}{dt} = \frac{\partial \overline{H}}{\partial t}
\end{equation}
completes the HAMVET process defined in \cite{ross-book}.

All of these necessary conditions can be easily checked by running DIDO using the single line command,
\begin{equation}\label{eq:callingDIDO}
[\textsf{cost}, \textsf{primal}, \textsf{dual}] = \text{\textsf{\textbf{dido}}\big(\textsf{problem}, \textsf{algorithm}\big)}
\end{equation}
where \textsf{cost} is the (scalar) value of $J$, \textsf{primal} is as defined in \eqref{eq:primal_def_1of2} and \eqref{eq:primal_def_2of2}, and \textsf{dual} is the \emph{\textbf{output}} structure of DIDO that packs the entirety of the dual variables in the following format:
\begin{subequations}\label{eq:dual.stuff}
\begin{align}
&\textsf{dual}.\textbf{Hamiltonian}    & \longleftarrow  &&&    \mathcal{H}[@ t]\\
&\textsf{dual}.\textbf{dyanmics}  & \longleftarrow  &&&    \blam(\cdot)\\
&\textsf{dual}.\textbf{events}      & \longleftarrow   &&&    \bnu \\
&\textsf{dual}.\textbf{path}& \longleftarrow  &&&    \bmu(\cdot)
\end{align}
\end{subequations}
\textbf{\emph{Note that the necessary conditions are not supplied to DIDO}}; rather, the code does all of the hard work in generating the dual information.  It is the task of the user to check the optimality of the computed solution, and potentially discard the answer should the candidate optimal solution fail the test.

\section{Overview of DIDO's Algorithm}\label{sec:overview}
From Secs.~\ref{sec:intro} and \ref{sec:DIDO-i/o}, it is clear that DIDO interacts with a user in much the same way as a direct method but generates outputs similar to an indirect method.  In this regard, \emph{\textbf{DIDO is both a direct and an indirect method}}.  \emph{Alternatively, DIDO is neither a direct nor an indirect method! } To better understand the last two statements and DIDO's unique suite of algorithms, recall from Sec.~\ref{sec:DIDO-i/o} that the totality of variables in an optimal control problem are,
$$\bxf, \buf, t_0, t_f, \bp, \blam(\cdot), \bmu(\cdot), \ \text{and }\ \bnu$$
Internally, DIDO introduces two more function variables $\bv(\cdot)$ and $\bomega(\cdot)$  by rewriting the dynamics and adjoint equations as,\footnote{This rewriting is done after a domain transformation; see Appendix~A  for details.}
\begin{subequations}
\begin{align}
\dot\bx  &= \bv; & \bv  &= \partial_{\boldsymbol\lambda} \overline{H} \\
-\dot\blam  &= \bomega; & \bomega  &= \partial_{\boldsymbol x} \overline{H}
\end{align}
\end{subequations}
The quantities $\bv$ and $\bomega$ are called \textbf{\emph{virtual and co-virtual}}  variables respectively. Thus, all nonlinear differential equations are split into linear differentials and nonlinear algebraic components at the price of adding two new variables.  The computational cost for adding these two variables is quite cheap; see Sec.~\ref{sec:misc}.\ref{sec:memory-reqs} for details.  The linear components are handled separately and efficiently through the use of $N \in \mathbb{N}$ Birkhoff basis functions\cite{lorentz,finden,wang,birk-koeppen,birk-TN}.  DIDO then focuses on solving the nonlinear equations with the linear equations tagging along for feasibility. Of all the equations, special attention is paid to the the Hamiltonian minimization condition; see \eqref{eq:HMC=NLP}.  This is because the gradient of the cost function is damped by a ``step size'' when associated with the gradient of the Hamiltonian\cite{ross:direct-shooting}.  Thus, unlike nonlinear programming methods that treat all variables the same way, DIDO's \textbf{\emph{Hamiltonian programming method}} treats different variables differently.  This is because DIDO generates processes that are centered around a Hamiltonian while a nonlinear programming algorithm is based on equations generated by a Lagrangian. Because \textbf{\emph{a Hamiltonian is not a Lagrangian}}, DIDO treats a state variable differently than a control variable. Similarly, a dynamics function is handled differently than a path-constraint function. These differences in treatments of different functions and constraints are in sharp contrast to nonlinear programming methods where all variables and functions are treated identically.\footnote{Although nonlinear programming algorithms may treat linear and nonlinear constraints differently, this concept is agnostic to the differences betweens state and control variables.} The mathematical details of this computational theory are described in Appendix~A.

In addition to treating the optimal control variables differently, DIDO also treats the Birkhoff pseudospectral problem differently for different values of $N$ and the ``state'' of the iteration. As noted in Sec.~\ref{sec:intro}, DIDO's algorithm is actually a suite of algorithms.
There are three major algorithmic components to DIDO's main algorithm.  A schematic that explains the three components is shown in Fig.~\ref{fig:DIDOMain3}.  The basic idea behind the three-component formula is to address the three different objectives of DIDO, namely, to be \emph{\textbf{guess-free, fast and accurate}}.\footnote{The word fast is used here in the sense of true-fast; i.e. fast that is agnostic to the specifics of computer implementations. A more precise meaning of fast is defined in Appendix B.}
%
%======================================================================================
   \begin{figure}[h!]
      \centering
      \framebox{\parbox{0.9\columnwidth}{
      \centering
      {\includegraphics[angle=0, width = 0.9\columnwidth]{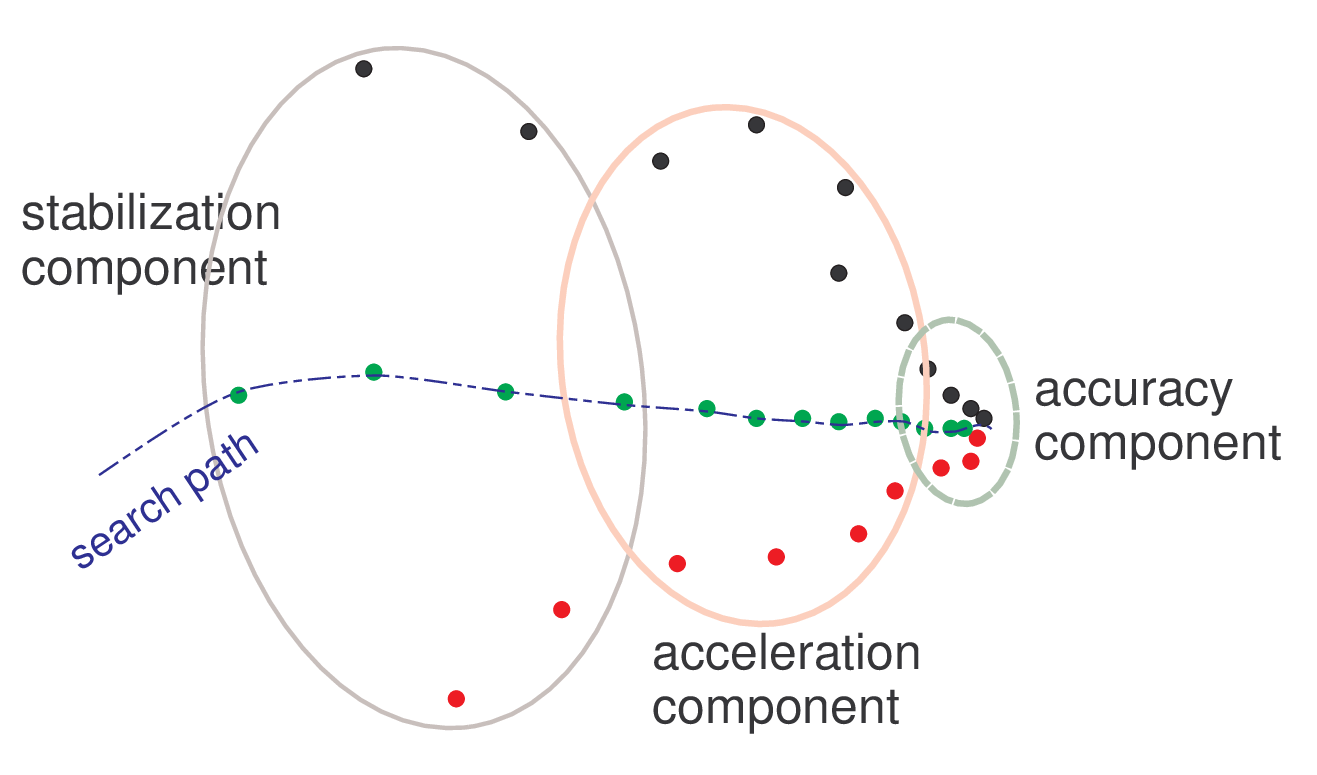}}
      \caption{\textsf{A schematic for the concept of the three main algorithmic components of DIDO.}}
    \label{fig:DIDOMain3}
    }}
   \end{figure}
%==========================================================================================
%
The functionality of the three major components of DIDO are defined as follows:
\begin{enumerate}
\item The stabilization component: The task of this component of the algorithm is to drive an ``arbitrary'' point to an ``acceptable'' starting point for the acceleration component.
\item The acceleration component: The task of this suite of algorithms is to rapidly guide the sequence of functional iterates to a capture zone of the accurate component.
 \item The accurate component: The task of this component is to generate and refine a solution that satisfies the precision requested by the user.
\end{enumerate}
% ------------------------
%
DIDO's main algorithm ties these three components together by monitoring their progress (i.e., ``state'' of iteration) and triggering automatic switching between components based on various criteria.  Details of all three components of DIDO are described in Appendix B.

\section{Best Practices in Using DIDO}\label{sec:best-practices}
Despite the multitude of internal algorithms that comprise DIDO's algorithm, \emph{\textbf{the only concept that is needed to use DIDO effectively is Pontryagin's Principle\cite{ross-book}}}. That is, the details of DIDO's algorithms are quite unnecessary to use DIDO effectively.  Best practices in producing a good DIDO application code in the least amount of time are:
\begin{enumerate}
\item Using Pontryagin's Principle to generate the necessary conditions for a given problem as outlined in Sec.~\ref{sec:DIDO-i/o}.
\item Determining if the totality of necessary conditions are free of concept/computational problems; for example, square roots, division by zero etc. If yes, then reformulating the problem as discussed in \cite{ross-book}, Chs.~1 and 3.
\item Identifying the subset of the necessary conditions that can be easily checked.  Examples are constant/linear costate predictions, switching conditions etc.
\item Scaling and balancing the equations.  A detailed process for scaling and balancing is described in \cite{scaling}.
\end{enumerate}
If a problem is not well-balanced, DIDO will execute several subalgorithms to generate an extremal solution. This will effectively bypass its main acceleration component. \emph{\textbf{Tell-tale signs of imbalanced equations are if the costates and/or Hamiltonian values are very large or very small; e.g., $10^{\pm9}$}.}

Once DIDO generates a candidate solution, it is critical for a user to \emph{\textbf{verify and validate (V\&V)}} the computed solution.  Best practices for post computation are:
\begin{enumerate}
\item Independently testing the feasibility of the computed solution. If the solution is not feasible, it is not optimal!
\item Testing the optimality of the computed solution using the dual outputs generated by DIDO.
\end{enumerate}
DIDO also contains a number of other computational components such as the DIDO Doctor Toolkit$^\text{TM}$ to assist the user in diagnosing and helping avoid certain common mistakes in coding and problem formulation. Additional features continue to be added, particularly when new theoretical advancements are made and/or new improvements in computational procedures are possible.

\section{Solving A Challenge Problem from Robotics}\label{sec:robotics}
One of the earliest applications of DIDO that went beyond its aerospace boundaries were problems in robotics\cite{lewis:nato}.  A motion planning problem with obstacle avoidance can easily be framed as an optimal control problem with path constraints\cite{lewis:cdc, infotech}.  A particular time-optimal motion planning problem for a \emph{\textbf{differential-drive robot}} can be framed as:
%
%============================================================================
\begin{align*}
%---------------------------------------------------
&\left.\begin{aligned}
\phantom{problem}
\X & = \real{3}   &\U & = \real{2}& \\
\bx &= (x, y, \theta)   &\bu &= (\omega_r, \omega_l) & \\
\end{aligned}\hspace{1.65cm} \right\} \ \text{\emph{\textbf{preamble}}}
%---------------------------------------------------
\\[0.5em]
%\]
%\vspace{-18pt} % USE THIS TO ALIGN!!
%%==========================================================
%\newlength{\leftside}
%\newlength{\rightside}
%\newcommand*{\leftterm}{}
%\newcommand*{\rightterm}{}
%\newcommand*{\term}[1]{$\displaystyle#1$}
%-----------------------------------------------------------
%\[
%
&\begin{aligned}
\renewcommand*{\leftterm}{\phantom{ J[\bx(\cdot), \bu(\cdot), t_0, t_f]}}
\renewcommand*{\rightterm}{\phantom{(\bx_0, x t_f) }}
\settowidth{\leftside}{\term{\leftterm}}
\settowidth{\rightside}{\term{\rightterm}}
%----------------------------------------------------------
%%---------------------------------------------------
%\begin{array}{rlrlrl}
%\X &=\real{3}   &\U &= \Real& \\
%\bx &= (x, y, v)  \quad &\bu &= \theta & \\
%\end{array}\\
%%---------------------------------------------------
\overbrace{(R)}^{\text{\normalsize \emph{\textbf{problem}}}} \left\{
\begin{array}{ll}
\text{Minimize}& \\[-1.20em] % lines up Minimize with J
& \left.\begin{aligned}
%\makebox[\leftside][l]{\term{\phantom{J[\bx(\cdot), \bu(\cdot), t_0,]} }} & \\[0.5em]
\makebox[\leftside][l]{\term{J[\bx(\cdot), \bu(\cdot), t_0, t_f] }} & \\[0.5em]
 \makebox[\leftside][r]{\term{:= t_f-t_0}} & \\ %+ \makebox[\rightside][l]{\term{Fsf}}\\
% \makebox[\leftside][r]{\term{E(\bx_0, \bx_f, t_0, t_f, \bp)}} &  \\[0.5em]
 % \makebox[\rightside][l]{\term{\int_{t_0}^{t_f}F(\bx(t), \bu(t), t)\ dt }}  \\
%  \makebox[\leftside][l]{\term{+ \int_{t_0}^{t_f}F(\bx(t), \bu(t), t, \bp)\ dt }} & \makebox[\rightside][l]{\term{\phantom{a}}}\\
 \end{aligned}\hspace{1.25cm} \right\} \ \text{\textbf{\emph{cost}}}\\ [2.7em]
\text{Subject to} &  %\begin{aligned}
%    \makebox[\leftside][r]{\term{\dot x}} &= \makebox[\rightside][l]{\term{v\sin\theta}}\\
%  \end{aligned}
\\[-1.75em]
&  \left.\begin{aligned}
%    \makebox[\leftside][r]{\term{\dot\bx}} &= \\
    \makebox[\leftside][l]{\term{ \dot x = \frac{\cos\theta}{2}(\omega_r + \omega_l)}} & \makebox[\rightside][l]{\term{\phantom{a}}} \\[0.5em]
    \makebox[\leftside][l]{\term{ \dot y = \frac{\sin\theta}{2}(\omega_r + \omega_l)}} & \makebox[\rightside][l]{\term{\phantom{a}}} \\[0.3em]
    \makebox[\leftside][l]{\term{ \dot\theta = c(\omega_r - \omega_l)}} & \makebox[\rightside][l]{\term{\phantom{a}}} \\
  \end{aligned}\right\}\ \text{\emph{\textbf{dynamics}}} \\ [3em]
  %
  %%
%&  \left.\begin{aligned}
%    \makebox[\leftside][r]{\term{\dot x}} &= \makebox[\rightside][l]{\term{v\sin\theta}} \\
%    \dot y &= v\cos\theta \\
%    \dot v &= g\cos\theta
%  \end{aligned}\right\} \quad \textsf{(dynamics)} \\[2em]
%  %
&  \left.\begin{aligned}
   % \makebox[\leftside][r]{\term{\be^L}}&\le \\
    \makebox[\leftside][l]{\term{ \bx(0) = (x^0, y^0, \theta^0) }} &
    \makebox[\rightside][l]{\term{\phantom{a}}}\\
    \makebox[\leftside][l]{\term{ \bx(t_f) = (x^f, y^f, \theta^f) }} &
    \makebox[\rightside][l]{\term{\phantom{a}}}\\
%    \makebox[\leftside][r]{\term{\be^L \le \be(\bx_0, \bx_f, t_0, t_f)}}&\le \makebox[\rightside][l]{\term{\be^U}}\\
%    \be^L &=  (0, 0)
  \end{aligned}\right\} \ \text{\emph{\textbf{events}}} \\[1em]
  &  \left.\begin{aligned}
%    (t_0, x_0, y_0, v_0) &=  \makebox[\rightside][l]{\term{(0, 0, 0, 0)}}\\
%   \makebox[\leftside][r]{\term{\bh^L}}
%   &\le \makebox[\rightside][l]{\term{\bh(\bx(t), \bu(t), t) \le \bh^U}}\\
   \makebox[\leftside][r]{\term{ (x-x_i)^2 + (y-y_i)^2 - r_i^2 \ge 0 }}
   & \makebox[\rightside][l]{\term{\phantom{a}}}\\[0.5em]
   \makebox[\leftside][r]{\term{ i= 1, \ldots, N_{obs} }}
   & \makebox[\rightside][l]{\term{\phantom{a}}}\\[0.5em]
    \makebox[\leftside][r]{\term{ \omega^L \le \omega_{r, l} \le \omega^U  }}
   & \makebox[\rightside][l]{\term{\phantom{a}}}\\[0.5em]
  \end{aligned}\right\} \ \text{\emph{\textbf{path}}} \\[1em]
%%----------------------------------------
%    &  \left.\begin{aligned}
%   \makebox[\leftside][r]{\term{\bK^L \le \bK(\bx(\cdot), \bu(\cdot), t, \bp) \le \bK^U}}
%   & \makebox[\rightside][l]{\term{\phantom{a}}}\\
%  \end{aligned}\right\} \ \text{\emph{\textbf{functional}}}
%%--------------------------------------
\end{array}
\right.
\end{aligned}
\end{align*}
%==================================================================
%
%
%The (kinematical) equations of motion of a differential drive robot are given by,
%%
%\begin{subequations}
%\begin{align}
%\dot x &= \frac{\cos\theta}{2}(\omega_r + \omega_l)  \\
%\dot y &= \frac{\sin\theta}{2} (\omega_r + \omega_l) \\
%\dot\theta & = c(\omega_r - \omega_l)
%\end{align}
%\end{subequations}
%%
%
where, $\bx := (x, y, \theta)$ is the state of a differential-drive robot, $\bu = (\omega_r, \omega_l)$ is the control and $c$ is a constant\cite{lynch-park:book}. Obstacles the robot must avoid during navigation are modeled as circular path constraints\cite{lewis:cdc} of radius $R_i$ centered at $(x_i, y_i), \ i = 1, \ldots, N_{obs}$.  The robot itself has a ``bubble'' of radius $R_0$; hence, $r_i:= R_i + R_0$.

Although it is very tempting to take the posed problem and code it up in DIDO, this is highly \emph{\textbf{not recommended}} as noted in Sec.~\ref{sec:DIDO-i/o}.  This is because \emph{DIDO is not an application software}; rather, it is a mathematical toolbox for solving a generic optimal control problem (see Problem~$(G)$ posed in Sec.~\ref{sec:intro}). Thus, a user is creating an ``app'' using DIDO as the ``operating system.'' Consequently, an analysis of the posed optimal control problem is \emph{\textbf{strongly recommended}} before coding. This advice is in sharp contrast to the brute-force advocacy of ``nonlinear programming methods for optimal control.''

\subsection{Step 1: Pre-Code Analysis}\label{sec:pre-code-analysis}
A pre-code analysis is essentially an analysis of the posed problem using Pontryagin's Principle\cite{ross-book}.  In following the ``HAMVET'' procedure enunciated in Sec.~\ref{sec:DIDO-i/o}, we first construct the Pontryagin Hamiltonian,
\begin{equation}
H(\blam, \bx, \bu):= \frac{\lambda_x \cos\theta + \lambda_y \sin\theta}{2}(\omega_r + \omega_l)
+ c \lambda_\theta (\omega_r - \omega_l)
\end{equation}
where $\blam := (\lambda_x, \lambda_y, \lambda_\theta)$ is an adjoint covector (costate).  The Lagrangian of the Hamiltonian is given by,
\begin{multline}
\overline{H}(\bmu, \blam, \bx, \bu) := H(\blam, \bx, \bu) + \sum_{i=1}^{N_{obs}} \mu_i (x-x_i)^2 + (y - y_i)^2 \\
+ \mu_r \omega_r + \mu_l \omega_l
\end{multline}
where, $\bmu := (\mu_1, \ldots, \mu_{N_{obs}}, \mu_r, \mu_l)$ satisfies the complementarity conditions,
\begin{subequations}
\begin{align}
\mu_{i} & \left\{\begin{array}{ccc}
  \le 0 & \ \text{if} \ \  & (x-x_i)^2 + (y - y_i)^2 = r_i^2 \\
  = 0 & \ \text{if} \ \ & (x-x_i)^2 + (y - y_i)^2  > r_i^2
\end{array}  \right. \label{eq:mu_i=comp}\\[0.5em]
\mu_{r,l} & \left\{\begin{array}{ccc}
  \le 0 & \ \text{if} \ \  & \omega_{r,l} = \omega^L \\
  = 0 & \ \text{if} \ \ & \omega^L < \omega_{r,l} < \omega^U \\
  \ge 0 & \ \text{if} \ \  & \omega_{r,l} = \omega^U
\end{array}  \right.\label{eq:mu=bb-maybe}
\end{align}
\end{subequations}
The adjoint equations are now given by,
\begin{subequations}\label{eq:adj4R}
\begin{align}
-\dot\lambda_x  &:= \frac{\partial \overline{H}}{\partial x}= \sum_{i=1}^{N_{obs}} 2 \mu_i (x-x_i) \\
-\dot\lambda_y  &:= \frac{\partial \overline{H}}{\partial y}= \sum_{i=1}^{N_{obs}} 2 \mu_i (y-y_i) \\
-\dot\lambda_\theta &:= \frac{\partial \overline{H}}{\partial \theta}= \frac{-\lambda_x \sin\theta + \lambda_y \cos\theta}{2}(\omega_r + \omega_l)
\end{align}
\end{subequations}
Hence, from \eqref{eq:adj4R} and \eqref{eq:mu_i=comp}, we get,
\begin{equation}\label{eq:VV-1}
\boxed{
(\dot\lambda_x, \dot\lambda_y) = (0, 0) \quad \text{if }  \mu_i = 0  \ \forall\ i = 1, \ldots, N_{obs}
}
\end{equation}
Equation~\eqref{eq:VV-1} is thus a special necessary condition for Problem~$(R)$. Any claim of an optimal solution to Problem~$(R)$ must be supported by \eqref{eq:VV-1} for mathematical legitimacy. \emph{Hence, before writing the first line of code, we expect DIDO to generate constant co-position trajectories $t \mapsto (\lambda_x, \lambda_y)$ whenever the robot is sufficiently far from all obstacles}.  Furthermore, over regions where $\mu_i \ne 0$, we expect
$t \mapsto (\lambda_x, \lambda_y)$ to be different. More importantly, because $t \mapsto \bmu$ is allowed to be a Dirac delta function (see \cite{ross-book} for details) the co-position trajectories may jump instantaneously. In addition, because $\mu_i(t) \le 0$, the jumps (if any) in the co-position trajectories are always positive (upwards) if $x(t) > x_i$ for $\lambda_x(t)$ and $y(t) > y_i$ for $\lambda_y(t)$. \emph{Hence, before any code is written, we expect DIDO to generate constant co-position trajectories with possible up-jumps (or down-jumps)}. Note that these are \emph{\textbf{all necessary conditions}} for optimality.  Stated differently, any candidate solution that cannot satisfy these conditions is certainly not optimal.  Consequently, when optimality claims are made on computed solutions without a demonstration of such necessary conditions, it is questionable, at the very least.

Next, from the Hamiltonian minimization condition, we get,
\begin{subequations}\label{eq:SA-maybe}
\begin{align}
\frac{\lambda_x \cos\theta + \lambda_y \sin\theta}{2}
+ c \lambda_\theta + \mu_r = 0 \\
\frac{\lambda_x \cos\theta + \lambda_y \sin\theta}{2}
- c \lambda_\theta + \mu_l = 0
\end{align}
\end{subequations}
If $\mu_r(t) = 0$ or $\mu_l(t) = 0$ over a nonzero time interval, then \eqref{eq:SA-maybe} must be analyzed for the existence of singular arcs.  If singular arcs are absent, then, from \eqref{eq:mu=bb-maybe} we get the condition that $t \mapsto (\omega_r, \omega_l)$ must be ``bang-bang'' and in accordance with the ``law:''
\begin{equation}\label{eq:VV-2}
\boxed{
\omega_{r,l}(t) =  \left\{\begin{array}{ccc}
   \omega^L & \ \text{if} \ \  & \mu_{r,l}(t) \le 0 \\
  \omega^U  & \ \text{if} \ \  & \mu_{r,l}(t) \ge 0
\end{array}  \right.
}
\end{equation}

Finally, combining the Hamiltonian value condition with the Hamiltonian evolution equation, we get the condition,
\begin{equation}\label{eq:VV-3}
\boxed{
\mathcal{H}[@t] = - 1
}
\end{equation}

\subsection{Step 2: Scaling/Balancing and Setting up the ``Problem of Problems''}
The second critical preparatory step in setting up a problem for DIDO is scaling and balancing the equations\cite{scaling}.  This aspect of the the problem set up is discussed extensively in \cite{scaling} and the reader is strongly encouraged to read and understand the principles and practices presented in this article. At this point, we simply note that it is often advantageous to ``double-scale'' a problem.  That is, the first scaling is done ``by-hand'' to reduce the number of design parameters (constants) associated with a problem\cite{ross-book}.  Frequently, canonical scaling achieves this task\cite{scaling}.  The second scaling is done numerically to balance the equations\cite{scaling}.

Once a problem is scaled to its lowest number of constants, it is worthwhile to set up a ``problem of problems''\cite{ross-book} even if the intent is to solve a single problem.  To better explain this statement, observe that although Problem~$(R)$ is a specific optimal control problem in the sense of its specificity with respect to Problem~$(G)$ defined in Sec.~\ref{sec:intro}, one can still create a menu of Problem~$(R's)$ by changing the number of obstacles $N_{obs}$, the location of the obstacles $(x_i, y_i)$, the size of the obstacles $R_i$, the size of the robot $R_0$, the exact specification of the boundary conditions $(\bx^0, \bx^f)$ etc.  This is the problem of problems.  The pre-code analysis and the necessary conditions derived in Sec.~\ref{sec:robotics}.\ref{sec:pre-code-analysis} are agnostic to these changes in the data. Furthermore, once a DIDO-application-specific code is set up, it is easy to change such data and generate a large volume of plots and results.  This volume of data can be reduced if the problem is first scaled by hand to reduce the number of constants down to its lowest level\cite{ross-book}. In any case, the new challenge generated by DIDO-runs is a process to keep track of all these variations.  This problem is outside the scope of this paper.   The problem that is in scope is a challenge problem, posed by D. Robinson\cite{robinson}.  Robinson's challenge comprises a clever choice of boundary conditions and the number, location and size of the obstacles.  The challenge problem is illustrated in Fig.~\ref{fig:Rsetup}.
%
%======================================================================================
   \begin{figure}[h!]
      \centering
      {\includegraphics[angle=0, width = 0.98\columnwidth]{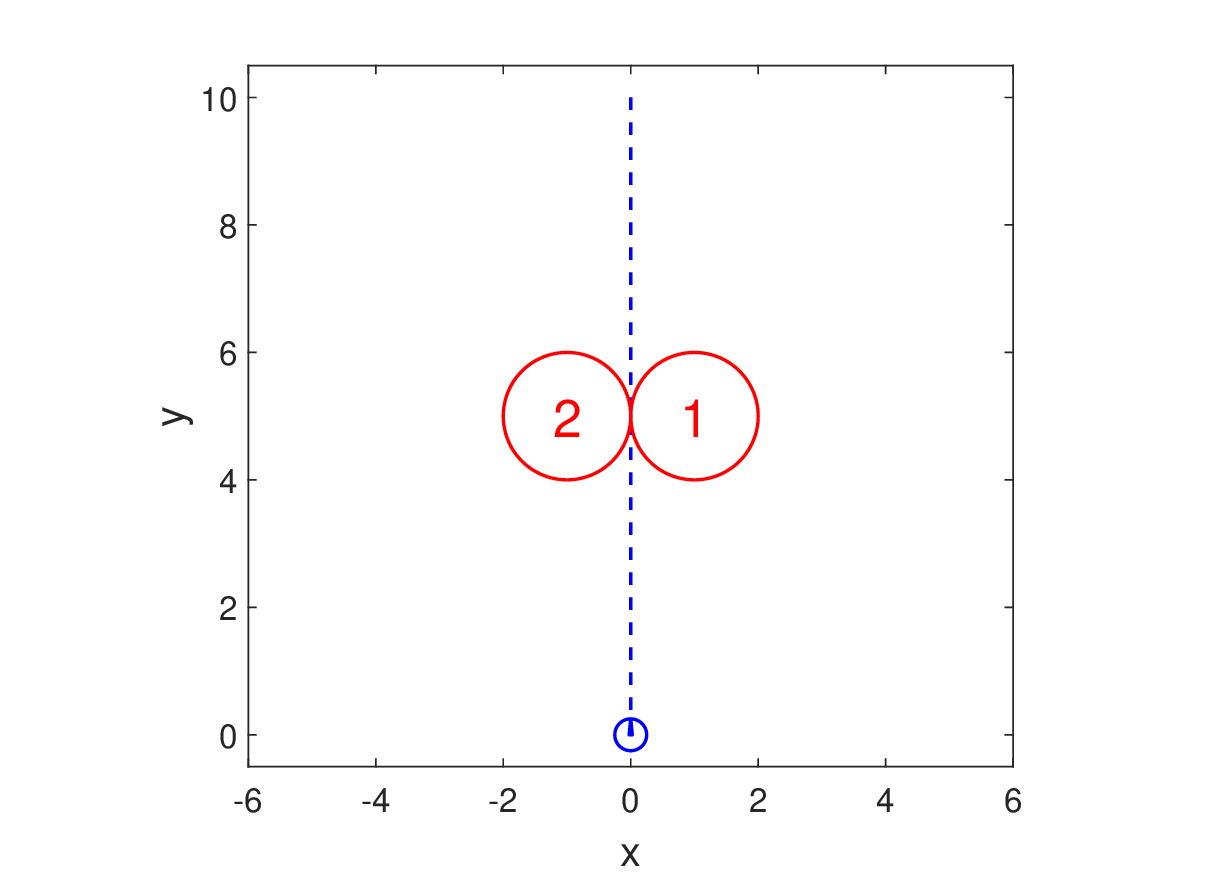}}
      \caption{\textsf{``Trick'' boundary conditions and obstacles set up for Problem~$(R)$.}}
    \label{fig:Rsetup}
   \end{figure}
%==========================================================================================
%

As shown in Fig.~\ref{fig:Rsetup}, two obstacles are placed in a figure-8 configuration and the boundary conditions for the robot are on opposite sides of the composite obstacle.  The idea is to ``trick'' a user in setting up a guess in the straight line path shown in blue.  In using this guess, an optimal control solver would get trapped by straight line solutions and would claim to generate an optimal solution by ``jumping'' over the single intersection point of the figure-8 configuration.  Attempts at ``mesh refinement'' would simply redistribute the mesh so that the discrete solution would continue to jump across the \emph{single point obstacle} in the middle. According to Robinson, all optimal control codes failed to solve this problem with such an initialization.

\subsection{Step 3: Generating DIDO's Guess-Free Solution}
\emph{Because DIDO is guess-free, it  requires no initialization.} Thus, one only defines the problem as stated and runs DIDO. It is clear from Fig.~\ref{fig:RsolD} that DIDO was able to provide the robot a viable path to navigate around the obstacles.
This is DIDO's Superior To Intuition$^{\text{TM}}$ capability.
%
%======================================================================================
   \begin{figure}[h!]
      \centering
      {\includegraphics[angle=0, width = 0.9\columnwidth]{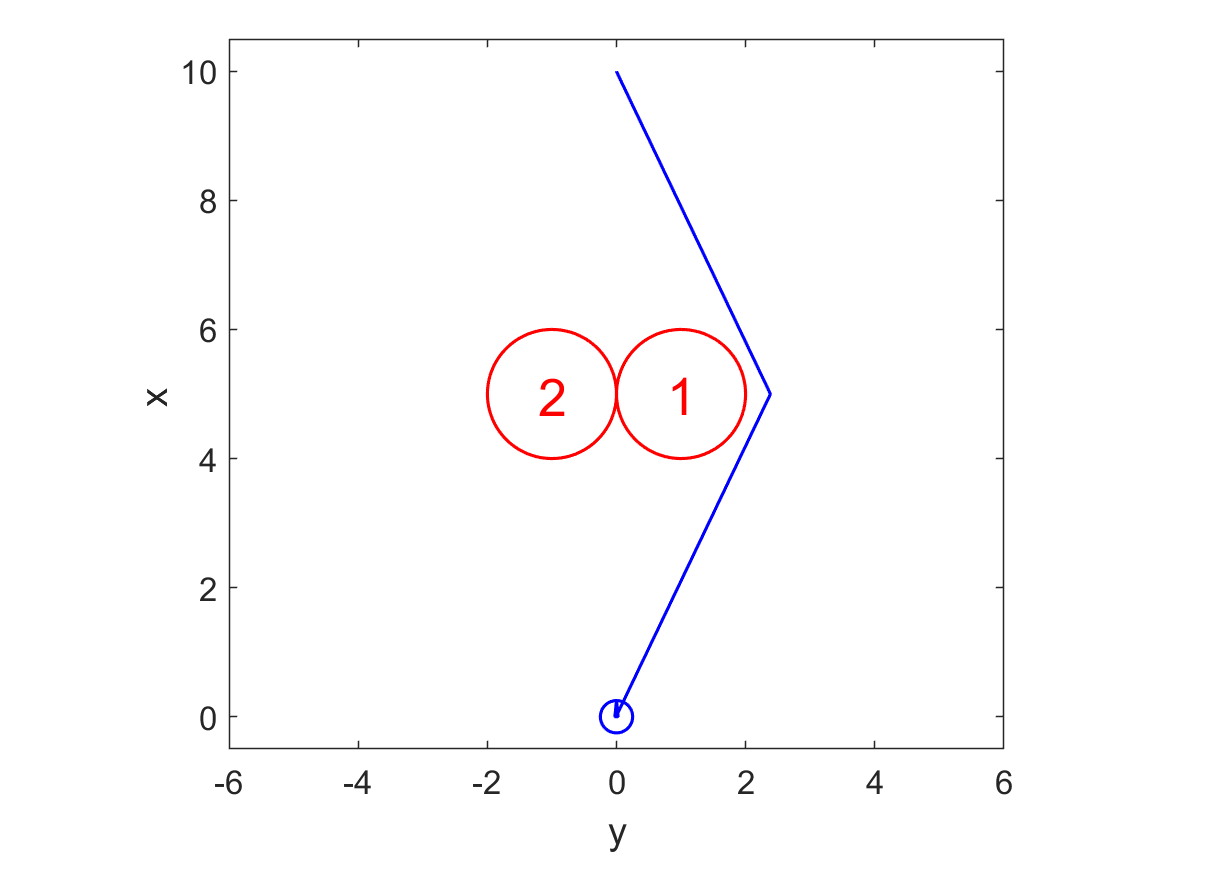}}
      \caption{\textsf{DIDO's solution to Problem~$(R)$ obtained via its default guess-free mode that illustrates its Superior To Intuition$^{\text{TM}}$ capability.}}
    \label{fig:RsolD}
   \end{figure}
%==========================================================================================
%
At this point, one can simply declare success relative to the posed challenge and ``collect the reward.'' Nonetheless, in noting that DIDO is purportedly capable of solving optimal control problems, it is convenient to use Problem~$(R)$ to illustrate how the ``path'' shown in Fig.~\ref{fig:RsolD} may have actually solved the stated problem.  In effect, there are (at least) two steps in this \emph{\textbf{verification and validation}} process (known widely as \emph{\textbf{V\&V}}):
\begin{enumerate}
\item Using a process that is independent of the optimization, show that the computed solution is feasible; and
\item Show that the necessary conditions (as many as possible\footnote{See, for example, \eqref{eq:VV-1}, \eqref{eq:VV-2} and \eqref{eq:VV-3}.}) are satisfied.
\end{enumerate}
In engineering applications, Step~1 is quite critical for safety, reliability and success of operations. Furthermore, it is also a critical feature of any numerical analysis procedure to ``independently'' check the validity of a solution.

\subsection{Step 4: Producing an Independent Feasibility Test}
To demonstrate the independent feasibility of a candidate solution to an optimal control problem, we simply take the computed control trajectory $t \mapsto \bu^\sharp$ and propagate the initial conditions through the ODE,
\begin{equation}\label{eq:ivp}
\dot\bx = \bff(\bx, \bu^\sharp(t)), \quad \bx(t_0) = \bx^0
\end{equation}
Let $t \mapsto \bx^\sharp$ be the propagated trajectory; i.e., one obtained by solving the initial value problem given by \eqref{eq:ivp}.  Then, $\bx^\sharp(\cdot)$  is the truth solution that is used to test if other constraints (e.g. terminal conditions, path constraints etc.) are indeed satisfied.

In performing this test for Problem~$(R)$, we first need the DIDO-computed control trajectory.  This is shown in Fig.~\ref{fig:Rcontrols}.
%
%======================================================================================
   \begin{figure}[h!]
      \centering
      {\includegraphics[angle=0, width = 0.8\columnwidth]{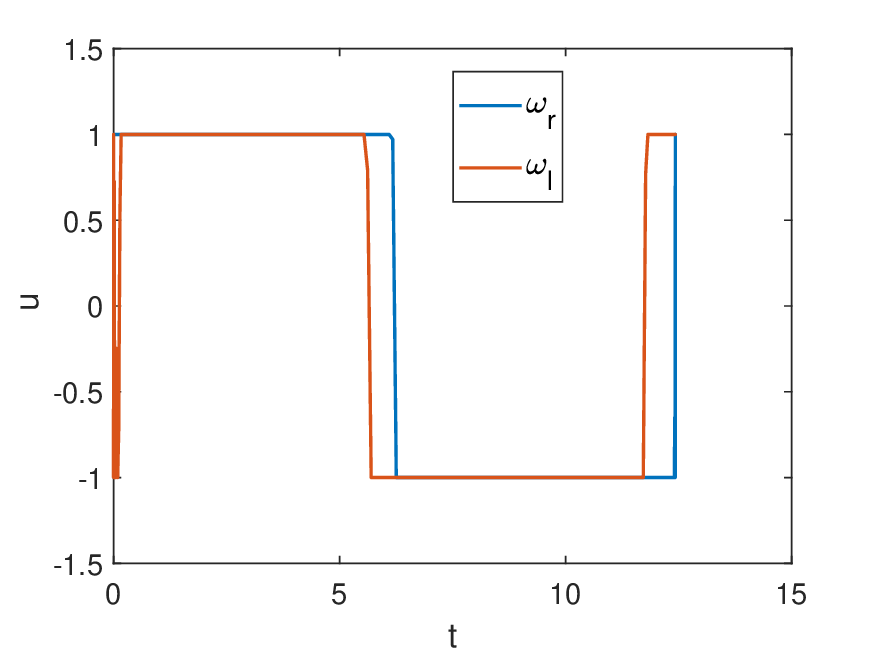}}
      \caption{\textsf{Candidate DIDO control trajectory to Problem~$(R)$.}}
    \label{fig:Rcontrols}
   \end{figure}
%==========================================================================================
%
It is clear that both controls take on ``bang-bang'' values with no apparent singular arcs.
%
%------------------
\begin{remark}
It is abundantly clear from Fig.~\ref{fig:Rcontrols} that DIDO is easily able to represent discontinuous controls. Consequently, the frequent lamentation in the literature that PS methods cannot represent discontinuities in controls is quite false.
\end{remark}
%--------------------
Using $t \mapsto (\omega_r, \omega_l)$ shown in Fig.~\ref{fig:Rcontrols} to propagate the initial conditions through the ``dynamics'' of the differential drive robot generates the solution shown in Fig.~\ref{fig:RstatesProp}.  The independent propagator used to generate Fig.~\ref{fig:RstatesProp} was \textsf{ode45}\footnote{Different propagators can generate different results; this aspect of solving \eqref{eq:ivp} is not taken into account in the ensuing analysis. } in MATLAB$^\circledR$.
%
%======================================================================================
   \begin{figure}[h!]
      \centering
      {\includegraphics[angle=0, width = 0.95\columnwidth]{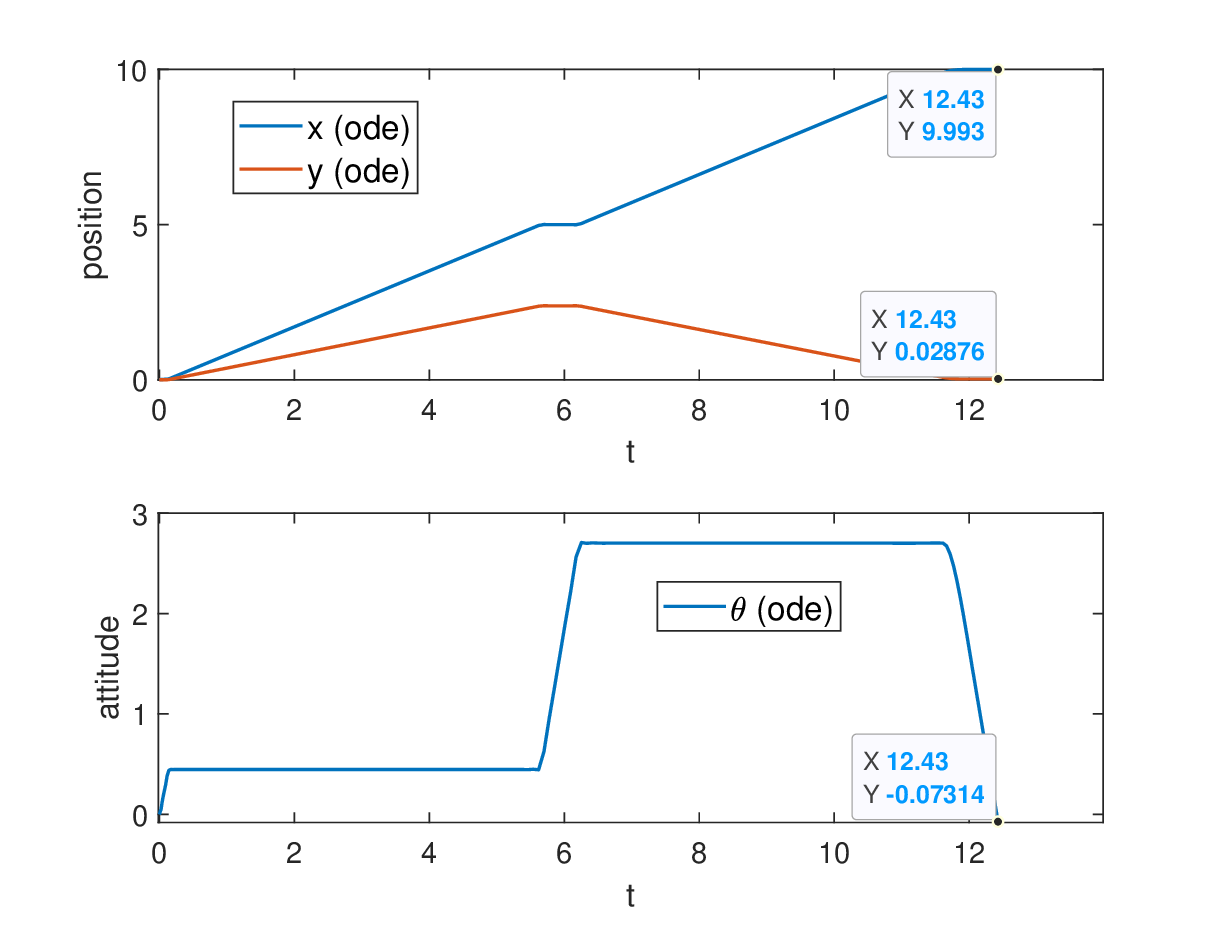}}
      \caption{\textsf{Propagated state trajectory using controls from Fig.~\ref{fig:Rcontrols}.}}
    \label{fig:RstatesProp}
   \end{figure}
%==========================================================================================
%
Also shown in Fig.~\ref{fig:RstatesProp} are the propagated values of the terminal states. The errors in the terminal values of the propagated states are the ``true'' errors.  They are significantly higher than those obtained by the optimization process as indicated in Table~\ref{tab:my-table}.
% \usepackage{booktabs}
\begin{table}[h!]
\centering
\begin{tabular}{@{}llll@{}}
\toprule
 & \multicolumn{1}{c}{x} & \multicolumn{1}{c}{y} & \multicolumn{1}{c}{$\theta$} \\ \midrule
Target & 10.0000 & 0.0000 & 0.0000 \\[0.5em]
%DIDO & 10.0000 & 0.0000 & 0.0000 \\
\begin{tabular}[c]{@{}l@{}}DIDO \\ (optimization)\end{tabular} & 10.0000 & 0.0000 & 0.0000 \\[1em]
\begin{tabular}[c]{@{}l@{}}Propagated \\ (true)\end{tabular} & 09.993 & 0.02876 & -0.07314 \\ \bottomrule
\end{tabular}
\caption{True accuracy vs optimization errors }
\label{tab:my-table}
\end{table}
%==============================================
The error values generated by any optimization process (including DIDO) must be viewed with a significant amount of skepticism because they do not reflect reality. Despite this apparently obvious statement, outrageously small errors that are not grounded in basic science or computational mathematics are frequently reported.  For instance, in \cite{betts-book}, a minimum time of $324.9750302$ seconds is reported (pg.~146). The purported accuracy of this number is ten significant digits.  Furthermore, a clock to measure this number must have a range of a few hundred seconds and an accuracy of greater than one microsecond!  To better understand the meaning of ``small'' errors (reasonable vs ridiculous), recall that a scalar control function $u(\cdot)$ is computationally represented as a vector, $U^N := (u_0, u_1, \ldots, u_N) $.  Let $U^N_\star$ be the optimal value of $U^N$ for a fixed $N$.  Then, a Taylor-series expansion of the cost function $U^N \mapsto J^N$ around $U^N_\star$ may be written as,
\begin{align}
J^N(U^N) &= J^N(U^N_\star) + \partial J^N(U^N_\star)\cdot(U^N - U^N_\star) \nonumber \\
         &\qquad +  (1/2) (U^N - U^N_\star) \cdot \partial^2 J^N(U^N_\star)\cdot(U^N - U^N_\star) \nonumber \\
         & \qquad \qquad + \text{higher order terms} \label{eq:derive-sqrt-error-1}
\end{align}
Setting $\partial J^N (U^N_\star) = 0$ in \eqref{eq:derive-sqrt-error-1}, we get,
\begin{align}
2 \left|J^N(U^N) - J^N(U^N_\star)\right| =
         & \left|\widehat{U}\cdot \partial^2 J^N(U^N_\star) \cdot \widehat{U} \right|\left\|U^N - U^N_\star \right\|^2 \nonumber \\
         & \qquad \qquad + \text{higher order terms} \label{eq:derive-sqrt-error-2}
\end{align}
where $\widehat{U}$ is a unit vector.
In a \emph{perfectly conditioned system}, the Hessian term in \eqref{eq:derive-sqrt-error-2} is unity; hence, in this best-case scenario we have\footnote{Equation~\eqref{eq:sqrt-error} is well-known in optimization; see, for example \cite{GMSW}.},
\begin{equation}\label{eq:sqrt-error}
\norm{U^N - U^N_\star} \sim \mathcal{O}\left(\sqrt{\big|J^N(U^N) - J^N(U^N_\star)\big|}   \right)
\end{equation}
Let $\epsilon_M$ be the machine precision.  Then from \eqref{eq:sqrt-error}, we can write,
\begin{equation}\label{eq:error-formula-UN}
\norm{U^N - U^N_\star}  > \mathcal{O}\left(\sqrt{\epsilon_M}\right) \sim 10^{-8}
\end{equation}
where, the numerical value in \eqref{eq:error-formula-UN} is obtained by setting $\epsilon_M \sim 10^{-16}$ (double-precision, floating-point arithmetic).

The vector $U^N_\star$ is only an ``$N^{th}$ degree'' approximation to the optimal function $u(\cdot)$.
The ``best'' convergence rate in approximation of functions is exponential\cite{atap}.  This record is held by PS approximations\cite{atap, Kang_2008_convergence,kang-rate} because they have near-exponential or ``spectral'' rates of convergence (assuming no segmentation, no knotted concatenation of low-order polynomials for mesh-refinement\cite{knots,auto-knots}, perfect implementation etc.).  In contrast, Runge-Kutta rates of approximation are typically\cite{hnw-ode} only $\mathcal{O}(4)$ or $\mathcal{O}(5)$ (fourth-order or fifth-order).  Assuming convergence to within a digit (very optimistically!) we get,
\begin{equation}\label{eq:-74u}
\left\| U^N(t) - u(t)\right\|  > 10\times\mathcal{O}\left(\sqrt{\epsilon_M}\right) \sim 10^{-7}
\end{equation}
where, we have abused notation for clarity in using $U^N(t)$ to imply an interpolated function.
In practice, the error may be worse than $10^{-7}$ (seven significant digits) if the derivatives of the data functions in a given problem (gradient/Jacobian) are computed by finite differencing.  Of course, none of these errors take modeling errors into account!  \emph{An independent feasibility test is therefore a quick and simple approach for \underline{estimating} the true feasibility errors of an optimized solution.  Any errors reported without such a test must be viewed with great suspicion!}\footnote{It is possible for the errors given by \eqref{eq:sqrt-error} (and hence \eqref{eq:-74u}) to be smaller in certain special cases. }

Given the arguments of the preceding paragraphs, it is apparent that a ``very small'' error that qualifies for the term ``very accurate'' control may be written as,
\begin{equation}\label{eq:-64u}
\|U^N(t) - u(t)\| \sim 10^{-6}
\end{equation}
Consequently, achieving an accuracy for $u(t)$ beyond six significant digits in a generic problem is highly unlikely.  Furthermore, recall that the estimate given by \eqref{eq:-64u} is based on perfect/optimistic intermediate steps!  Therefore, practical errors may be worse than $10^{-6}$. Regardless, to achieve high accuracy, the problem must be well conditioned.  To achieve better conditioning, the problem must be properly scaled and balanced regardless of which optimization software is used.  See \cite{scaling} for details.  Providing DIDO a well-scaled and balanced problem also triggers the best of its true-fast suite of algorithms.
%
%-------------
\begin{remark}
The optimistic error estimate given by \eqref{eq:-74u} or \eqref{eq:-64u} does not mean greater accuracy is not possible.  The main point of \eqref{eq:-64u} is to treat ``accuracy'' claims of computed solutions with great caution.   It is also important to note that the practical limit given by \eqref{eq:-64u} does not preclude the possibility of achieving highly precise solutions.  See, for example, \cite{Kepler-micro-slew} where accuracies of the order of $10^{-7}$~radians were achieved in pointing the Kepler space telescope.
\end{remark}
%-------------

\subsection{Step 5: Testing for Optimality}
This step closes the loop with Step 1.  In the pre-code analysis step, we noted (see \eqref{eq:VV-1}) that the co-position trajectories must be constants with potential jumps.  Recall that (see \eqref{eq:callingDIDO}) DIDO automatically generates all the duals as part of its outputs. Thus, the dual variables that are of interest in testing \eqref{eq:VV-1} is contained in \textsf{dual}.\textbf{dynamics}. A plot of the costate trajectories generated by DIDO is shown in Fig.~\ref{fig:Rlambdas}.
%
%======================================================================================
   \begin{figure}[h!]
      \centering
      {\includegraphics[angle=0, width = 0.8\columnwidth]{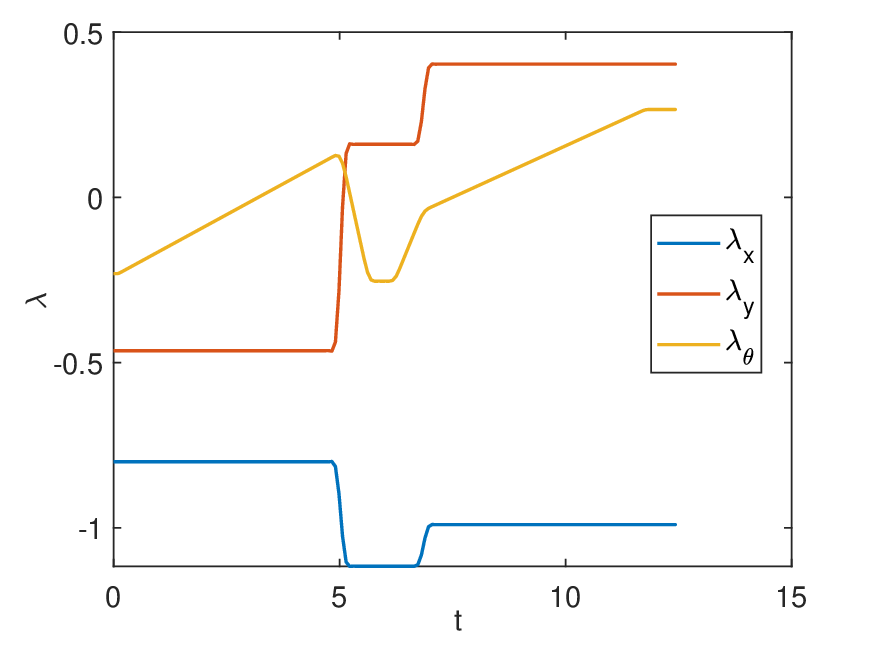}}
      \caption{\textsf{DIDO-generated costate trajectories for Problem~$(R)$.}}
    \label{fig:Rlambdas}
   \end{figure}
%==========================================================================================
%
It is apparent from this figure that the co-position trajectories $t \mapsto (\lambda_x, \lambda_y) $ are indeed constants with jumps over two spots. If $t_c$ is the time when a jump occurs, then it follows that $\lambda_y(t)$ must jump upwards because $y(t_c) > y_i$ where $t_c$ is the point in time when the robot ``bubble'' just touches the obstacle bubble. In the case of the $x$-coordinates, we have $x(t_c) < x_i$ at the first touch of the bubbles; see Fig.~\ref{fig:RsolD}. Hence, $\lambda_x(t)$ must jump downward, which it indeed does in Fig.~\ref{fig:Rlambdas}.  Subsequently, at the second touch of the bubbles $x(t_c) > x_i$; hence,  $\lambda_x(t)$ jumps upwards.

Because DIDO also outputs $\bmu(\cdot)$ in \textsf{dual}.\textbf{path}, it is a simple matter to test \eqref{eq:VV-2}.  The path covector trajectory $\mu_l(\cdot)$ is plotted in Fig.~\ref{fig:RswitchProof} over the graph of $\omega_l(\cdot)$.  It is clear from this plot that $\mu_l(\cdot)$ is indeed a switching function for $\omega_l(\cdot)$.
%
%======================================================================================
   \begin{figure}[h!]
      \centering
      {\includegraphics[angle=0, width = 0.8\columnwidth]{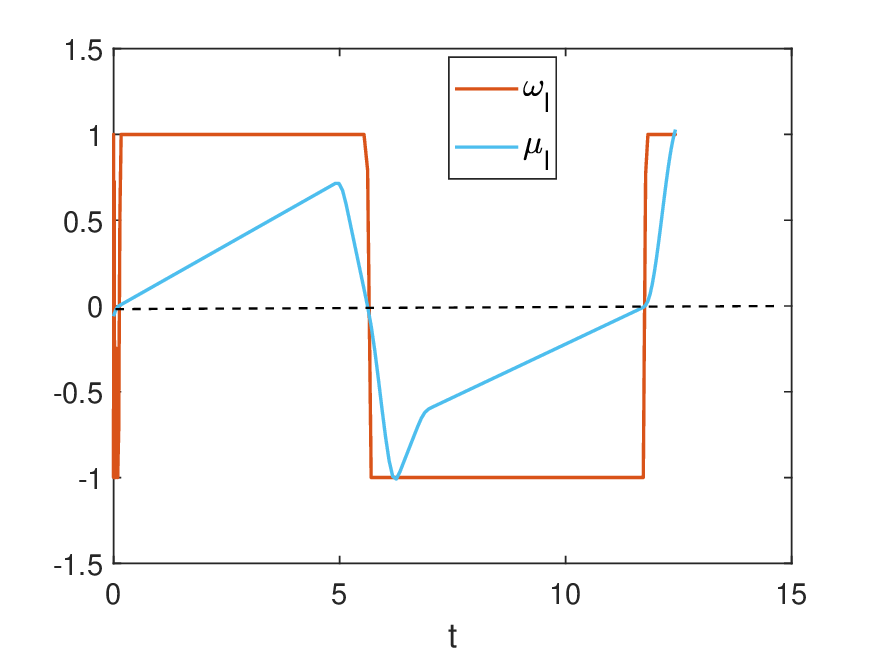}}
      \caption{\textsf{Verification of \eqref{eq:VV-2} for the $(\omega_l(\cdot), \mu_l(\cdot))$ pair generated by DIDO.}}
    \label{fig:RswitchProof}
   \end{figure}
%==========================================================================================
%
The same holds for $\omega_r$ and $\mu_r$. This plot is omitted for brevity.

Because there are two obstacles in the challenge problem (modeled as two inequalities), \textsf{dual}.\textbf{path} also contains $\mu_1(\cdot)$ and $\mu_2(\cdot)$ that satisfy the complementarity conditions given by \eqref{eq:mu_i=comp}.  It is apparent from the candidate solution shown in Fig.~\ref{fig:RsolD} that if this trajectory is optimal, we require $\mu_2(t) = 0$ over the entire time interval $[t_0, t_f]$.  That this is indeed true is shown in Fig.~\ref{fig:Rmu+lam}.  Also shown in Fig.~\ref{fig:Rmu+lam} is the path covector function $\mu_1(\cdot)$ generated by DIDO.
%
%======================================================================================
   \begin{figure}[h!]
      \centering
      {\includegraphics[angle=0, width = 0.8\columnwidth]{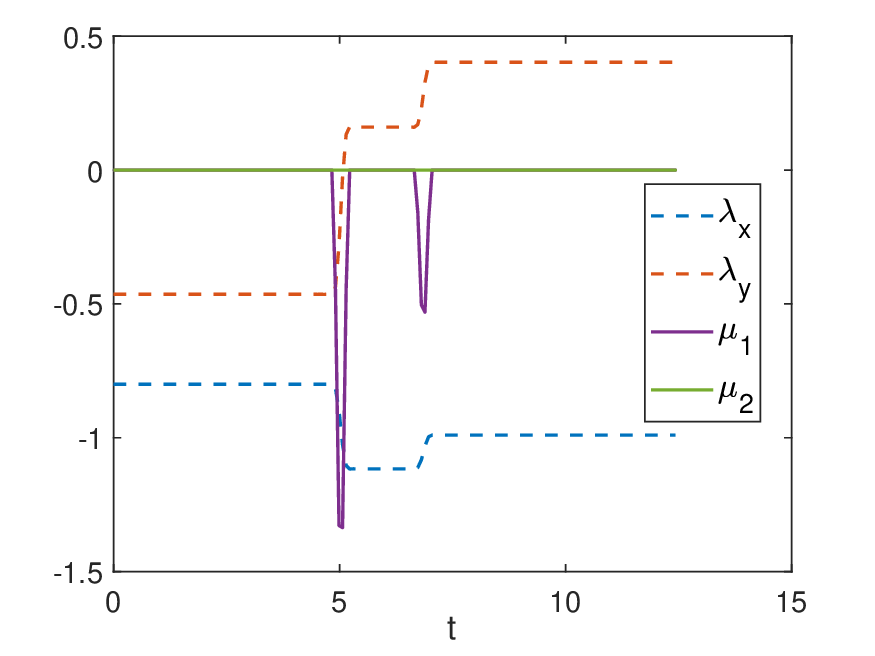}}
      \caption{\textsf{DIDO predicts $\mu_1(t)$ to be the sum of two impulse functions centered at the jump points of the co-position trajectories $t \mapsto (\lambda_x, \lambda_y) $.}}
    \label{fig:Rmu+lam}
   \end{figure}
%==========================================================================================
%
Note first that $\mu_1(t) \le 0\ \forall\ t$ as required by \eqref{eq:mu_i=comp}.  This function is zero everywhere except for two ``sharp peaks''
that are located exactly where the co-position trajectories jump. These sharp peaks are approximations to atomic measures\cite{ross-book}.  As noted in \cite{ross-book}, page 141, we can write $\bmu(\cdot)$ as the sum of two functions: one in $L^\infty$ and another as a finite sum of impulse functions.  From this concept, it follows that
\begin{equation}\label{eq:mu2=delta}
\mu_1(t) = \sum_j \eta_j \delta_D(t-t_j)
\end{equation}
where $\delta_D$ is the Dirac delta function centered at $t_j$ and $\eta_j$ is finite, the value of the impulse\cite{ross-book}.  From \eqref{eq:adj4R}, \eqref{eq:mu2=delta} and the signs of $(x(t_j) - x_1)$ and $(y(t_j) - y_1)$ (see Fig.~\ref{fig:RsolD}) it follows that $\lambda_x(t)$ must jump downwards first and upwards afterwards while $\lambda_y(t)$ must jump upwards at both touch points.  This is exactly the solution generated by DIDO in Fig.~\ref{fig:Rmu+lam}. \textbf{\emph{A useful teaching aspect of Fig.~\ref{fig:Rmu+lam} is its ability to explain difficult theoretical aspects of measure theory\cite{clarke-2013book} that are of practical value in understanding the fundamentals of optimal control theory}}.

Yet another verification of the optimality of the computed trajectory is shown in Fig.~\ref{fig:H}.
%
%======================================================================================
   \begin{figure}[h!]
      \centering
      {\includegraphics[angle=0, width = 0.8\columnwidth]{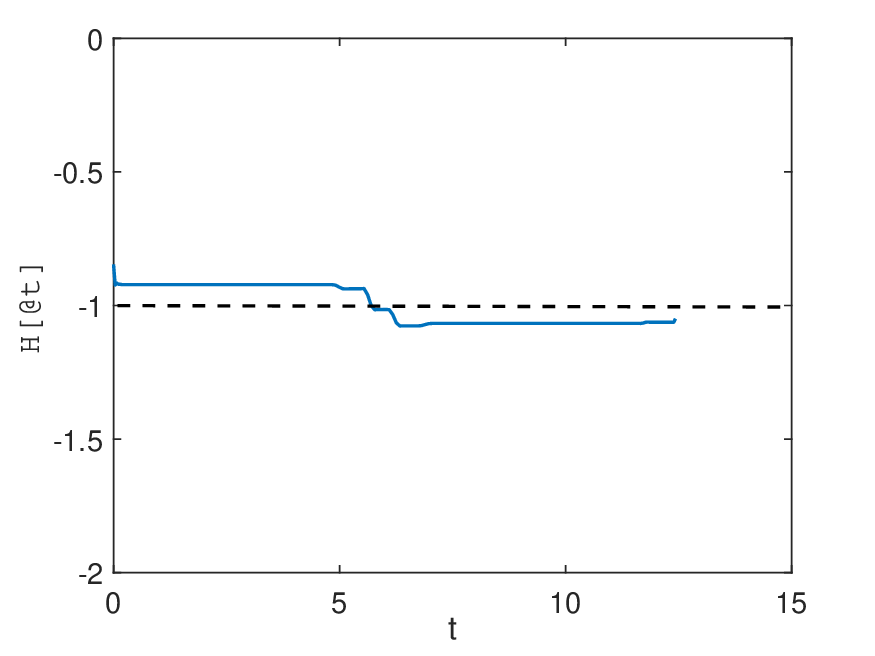}}
      \caption{\textsf{DIDO-computed evolution of the lower Hamiltonian $t \mapsto \mathcal{H}[@t]$ is on average a constant and equal to $-1$.}}
    \label{fig:H}
   \end{figure}
%==========================================================================================
%
According to \eqref{eq:VV-3}, the lower Hamiltonian must be a constant and equal to $-1$. DIDO's computation of the Hamiltonian evolution equation is not perfect as apparent from Fig.~\ref{fig:H}.  There are several reasons for this\cite{lncis}. With regards to Fig.~\ref{fig:H} specifically, one of them is, at least, in part because the computations of the Dirac delta functions in DIDO not perfect as evident from Fig.~\ref{fig:Rmu+lam}.

\section{Miscellaneous}\label{sec:misc}
\textbf{\emph{DIDO's algorithm as implemented in DIDO is not perfect!}} 
In recognition of this distinction, it should not be surprise to any user that DIDO has bugs.
It should also not be a surprise to any reader that DIDO and/or PS methods get frequently compared to other software, methods etc.  In this section we address such miscellaneous items.

\subsection{Known Bugs}
In general, DIDO will run smoothly if the coded problem is free of singularities, well posed, scaled and balanced. Both beginners and advanced users can avoid costly mistakes in following the best practices as outlined in Sec.~\ref{sec:best-practices}, together with using the DIDO Doctor Toolkit,$^{\text{TM}}$ and a preamble in coding the DIDO files.   Even after exercising such care, a user may inadvertently make a mistake that is not detected by the DIDO Doctor Toolkit.  In this case, DIDO may exit or quit with cryptic error messages. Examples of these are:
\subsubsection{Grid Interpolant Error}
This problem occurs when a user allows the possibility of $t_f - t_0 = 0$ in the form of the bounds on the initial and final time.  The simple remedy is to ensure that the bounds are set up so that $t_f - t_0 > 0$.

\subsubsection{Growthcheck Error}
This is frequently triggered because of an incorrect problem formulation that the DIDO Doctor Toolkit is unable to detect.  Best remedy is to check the coded problem formulation (and ensure that it is the same as one conceived). Example: if ``$x$'' is coded as ``$y$.''

\subsubsection{SOL Error}
This is not a bug; rather, it is simply as an error message where DIDO has gone wrong so badly that it is unable to provide any useful feedback to a user.  The best remedy for a user is to start over or to go back to a prior version of his/her files to start over.

\subsubsection{Hamiltonian Error}
This is neither a bug nor an error but is listed here for information. Consider the case where the lower Hamiltonian is expected to be a constant.  The computed Hamiltonian may undergo small jumps and oscillations around the expected constant.  These jumps and oscillations are amplified for minimum time problems and/or when one or more of the path covector functions contain Dirac-delta functions as evident in Fig.~\ref{fig:Rmu+lam}. When these are absent, the lower Hamiltonian computed by DIDO is reasonably accurate; see \cite{ross-book} for additional details on impulsive path covector functions and for the meaning of the various Hamiltonians.

\subsection{Comparison With Other Methods and Software}
It is forgone conclusion that every author of a code/method/software will claim superiority.  In maintaining greater objectivity, we simply identify \textbf{\emph{scientific facts}} regarding DIDO, PS theory and its implementations.

\subsubsection{PS vs RK; PS vs PS}
DIDO is based on PS discretizations.  The convergence of a PS discretization is spectral; i.e., almost exponentially fast\cite{atap, Kang_2008_convergence,kang-rate}. Convergence of Runge-Kutta (RK) discretizations are typically $\mathcal{O}(4)$.  \emph{\textbf{Consequently, it is mathematically impossible for a fourth-order method to converge faster than a higher-order method, particularly an ``exponentially fast'' method}}. Claims to the contrary are based on a combination of one or more of the following:
\begin{enumerate}[i)]
\item Not all PS discretizations are convergent just as not all RK methods are convergent\cite{hnw-ode}.  Unfortunately, even a convergent PS method may be inappropriate for a given problem\cite{PSReview-ARC-2012} because of mathematical technicalities associated with inner products and boundary conditions\cite{advances}. By using an inappropriate PS discretization over a set of numerical examples, it possible to show that it performs poorly or similarly to a convergent RK method\cite{betts-comments}.
\item The ``exponential convergence rate'' of a PS discretization is based on the order of the interpolating polynomial: higher the order, the better the approximation\cite{atap}. Consequently, an RK method can be easily made to outperform a sufficiently low-order PS discretization that is based on highly-segmented mesh refinements; see \cite{birk-koeppen} for details.
\item In a correct PS discretization, the control is \emph{\textbf{not}} necessarily represented by a polynomial\cite{PSReview-ARC-2012}; this fact has been known since at least 2006\cite{TAC:linearizable}.  It is possible to generate a bad PS discretization by setting the control trajectory to be a polynomial. In this case, even an Euler method can be made to outperform a PS discretization.
\item Low order PS discretizations are claimed to outperform high-order PS methods based on the assumption that a high-order method exhibits a Gibbs phenomenon.  It is apparent from Fig.~\ref{fig:Rcontrols} that there is no Gibbs phenomenon in a properly implemented PS discretization.

\end{enumerate}

\subsubsection{Implementation vs Theory; Theory vs Implementation}
PS discretizations are deceptively simple. Its proper implementation requires a significant amount of computational finesse\cite{wang,birk-TN,atap,bogaert,GLR,spec-diff-twist}.  For instance, in Lagrange PS methods (see Fig.~\ref{fig:PStypes} in Appendix~A), it is very easy to compute Gaussian grid points\cite{bogaert,GLR} poorly, or generate differentiation matrices with large round-off errors\cite{spec-diff-twist}, or or implement closure conditions\cite{advances,acc:hybrid} incorrectly.  In doing so, it is easy to show poor numerical performance of theoretically convergent PS discretizations.  Furthermore it is easy to show large run times of fast PS discretizations by simply patching it with generic nonlinear programming solvers.  It is also easy to claim poor computational performance of PS theory by using DIDO clones or other poor implementations of PS techniques\cite{betts-comments}.  It is also easy to show poor performance of DIDO by improper scaling or balancing or other inappropriate use of the toolbox.

To maintain scientific integrity, it is important to compare and contrast performance of trajectory optimization methods in a manner that is agnostic to its implementation and/or computational hardware. For instance, in static optimization, it is a scientific fact that a Newton method converges faster than a gradient method irrespective of its implementation on a computer\cite{GMSW}.  This is because computational speed is defined precisely in terms of the order of convergence and not in terms of computational time.  Furthermore, each Newton iteration is expensive relative to a gradient method because it involves the computation of a Hessian and solving a system of linear equations\cite{GMSW}. This fact remains unchanged no matter the implementation or the computer hardware.  Similar to its static optimization counterpart, it is indeed possible to perform computer-agnostic analysis of dynamic optimization methods; see \cite{ross:direct-shooting} for an example of a mathematical procedure for such an analysis.  \emph{\textbf{This topic in trajectory optimization is mostly uncharted and remains a wide open area of research}}. Nonetheless, there is no shortage of papers that use ``alternative facts'' and/or compare selective performance of codes as part of a self-fulfilling prophecy.

\subsubsection{DIDO $\gtreqqless$ PS}
It is quite possible for an algorithm to perform better than theory (e.g. simplex vs ellipsoid algorithm) or for a code to perform worse than theory.  The latter case is also true if a theory is implemented incorrectly, poorly or imperfectly. Consequently, using a code to draw conclusions about a theory is not only unscientific, it may also be quite misleading.  DIDO is not immune to this fact. Furthermore, because PS discretization is only an element of DIDO's algorithm, it is misleading to draw conclusions about PS theory using DIDO. This fact is strengthened by the point noted earlier that even DIDO's algorithm as implemented in DIDO is not perfect!  In fact, even the implementation is not complete! Despite all these caveats, it is still very easy to refute some of the claims made in the literature by simply running DIDO for the purportedly hard problems.  For example, the claims made in \cite{betts-comments} were easily refuted in \cite{marshDAE} by simply using DIDO.

\subsubsection{DIDO vs DIDO's algorithm}
In principle, DIDO's algorithm can generate ``near exact'' solutions;\footnote{The statement follows from the fact that DIDO's algorithm is grounded on the convergence property of the Stone-Weierstrass theorem and approximations in Sobolev spaces; see \cite{PSReview-ARC-2012}.} however, its precision is machine limited by the fundamental ``square-root equation'' given by \eqref{eq:sqrt-error}.  In practice, the precision is further limited than the square-root formula because of other machine and approximation errors (see \eqref{eq:-64u}).  In addition, DIDO's algorithm, as implemented in DIDO, the toolbox, is practically ``tuned'' to generate the best performance for ``industrial strength'' problems.  This is because the toolbox,  is designed to work across the ``space of all optimal control problems.'' Consequently, it is not too hard to design an academic problem that shows a ``poor performance'' of DIDO, the toolbox, even though DIDO's algorithm may be capable of significantly superior performance.

\subsubsection{First Facts: DIDO vs Other Software}
As the old adage (popularized by John F. Kennedy) goes, \emph{success has many fathers but failure is an orphan}.  Invariably, every author will claim something ``first'' or something ``better'' on a successful idea.  Fortunately, something first is easily verifiable.  The following milestones document some of DIDO's firsts:
\begin{enumerate}[i)]
\item The first PS optimal control solver was DIDO (this includes the pre-2001 version\cite{fr1}).  All optimal control solvers prior to DIDO were based on RK-type methods\cite{betts-survey}.  This fact is evident in the highly-regarded 1998 survey paper\cite{betts-survey} where there is no mention of PS methods;\footnote{Because survey papers typically look backwards, they have the unfortunate side-effect of being obsolete the moment they get published.} see also the 2001 book\cite{betts-book} for an absence of any discussion on PS optimal control theory.

\item The 2001 version of DIDO was the first object-oriented, general-purpose optimal control software\cite{ross-book,PSReview-ARC-2012,prague}. Prior to DIDO 2001, users had to provide files and variables in a convoluted fashion to nonlinear programming solvers; see, for example, Appendix~A.3 of \cite{betts-book}.  After it attained widespread publicity in 2006-07 (see Fig.~\ref{fig:ZPMcollage}), DIDO's format and structures (see Sec.~\ref{sec:DIDO-i/o}) were rapidly cloned\cite{ross-book}. Subsequently, the object-oriented methodology of DIDO was adopted as a computer-programming standard with varying modifications.

\item The first flight-proven general-purpose optimal control solver was the 2006 version of DIDO\cite{zpm:NASA-report,zpm:first,SIAMpg1}.  While other trajectory codes have been in the work-flow at NASA\cite{POST2ref}, they have all been special purpose tools. DIDO is the first general-purpose optimal control solver to be inserted in flight operations. Incidentally, it has also served as discovery tool for other NASA firsts; see the Feature Article in \textit{IEEE Spectrum}\cite{TRACE:IEEE}.

\item The 2008 version of DIDO was the first guess-free, general-purpose optimal control solver\cite{guess-free}.  To the best of the author's knowledge, DIDO remains the only such software well after a decade!

\item The first general-purpose optimal control solver to be embedded on a microprocessor was the 2011 version of DIDO's algorithm.  See \cite{PSReview-ARC-2012} for details on the computer architecture, embedding methodology, proof-of-concept and run-time performance improvements achieved by embedding.

\end{enumerate}

\subsection{Computer Memory Requirements}\label{sec:memory-reqs}
From \eqref{eq:PlamDN} in Appendix A, it follows that the memory requirements for the variables is $\mathcal{O}(N)$, where $(N+1)$ is the number of points that constitute the grid (in time).  This estimate can be refined further to $\mathcal{O}(nN)$ where $n$ is the number of continuous-time variables.  For instance, for Problem $(\bG)$, $n \sim (4N_x + N_u + N_h)$. Assuming $8$ bytes per variable, the memory requirements for variables can be written as
$\mathcal{O}(nN \times 10)$ bytes with $8$ replaced by $10$. Thus, for example, for $N=1,000$, the memory estimate can be written as $10n$ KB or $10$~KB per continuous variable.  Clearly, the number of variables is not a serious limiting factor for DIDO's algorithm.  This is part of the reason why it was argued in \cite{millionPT} that the production of solutions over a million grid points was more connected to numerical science than computer technology.

Carrying out an analysis similar to that of the previous paragraph, it is straightforward to show that the number of constraints is also not a significant limiting factor to execute DIDO's algorithm.  By the same token, the gradients of the nonlinear functions (see \eqref{eq:PlamDN}) are also not memory-intensive because Hamiltonian programming naturally exploits the separability property inherent in $(P^{\lambda, D, N})$. This conclusion carries over to Problem $(\bG)$ as well (see also \cite{birk-TN}).  Thus, the dominant sources of memory consumption are the matrices $\bA$ and $\bA^\star$ introduced in \eqref{eq:AXVN=0} and \eqref{eq:Astar-def} respectively (see Appendix A).  In choosing discretizations that satisfy Theorem~\ref{theorem:tunnel}, the memory requirements for the linear system given in \eqref{eq:PlamDN} is $2 N^2$ variables, and hence $2N_x N^2$ for solving Problem~$(\bG)$.  In using this number and the analysis of the previous paragraph, it follows that the memory cost for a grid size of $250$ points is about $1$~MB per state variable (and independent of the number of other variables). For $1,000$ grid points, the memory requirement is about $10$ MB per state variable.
%Consequently, generating solutions even for a million grid points is well within reach\cite{millionPT}.

Because its memory requirements are relatively low and convergence is fast, it follows that DIDO's algorithms can be easily embedded on a low end microprocessor. Specific details on a generic computer architecture for embedded optimal control are described in \cite{PSReview-ARC-2012}. As noted earlier, it is important to note the critical difference between DIDO's algorithms and DIDO, the optimal control toolbox.  The latter is not designed for embedding while the former is embed-ready. A proof-of-concept in embedding was shown in \cite{PSReview-ARC-2012}.  Consequently, it is technically feasible to generate real-time optimal controls by coordinating the Nyquist frequency requirements of digitized signals to the Lipschitz frequency production of optimal controls\cite{ross-book}.  The mathematical foundations of this theory along with some experimental results are described in \cite{ross:optimal-feedback} and \cite{ross:feedback-stability}; see also \cite{ross-book}, Sec.~1.5 and research articles on the Bellman pseudospectral method\cite{bellman-conf,bellman-low-t}.

\section{Conclusions}

\emph{\textbf{DIDO's solution to the ``space station problem'' that created headlines in 2007 is now widely used by other codes to benchmark their performance}}. The 2001--2007 versions of DIDO have been extensively cited as a model for new software.  The early versions of the software were indeed based on the admittedly  naive approach of patching nonlinear programming solvers to pseudospectral discretizations.  Later versions of the optimal control toolbox execute a suite of fundamentally different and multifaceted algorithms that are focused on being guess-free, true-fast and verifiably accurate.
Despite major internal changes, all versions of DIDO have remained true to its founding principle: to facilitate a fast and quick process in taking a concept to code using nothing but Pontryagin's Principle as a central analysis tool.  In this context, DIDO as a toolbox is more about a minimalist's approach to solving optimal control problems rather than pseudospectral theory or its implementation.  This is the main reason why \textbf{\emph{the best way to use DIDO is via Pontryagin's Principle!}}\footnote{DIDO's philosophy of \emph{analysis before coding}, as exemplified in Sec.~\ref{sec:robotics}, is in sharp contrast to the traditional ``brute-force'' direct method. The widely-advertised advantage of traditional direct methods is that a user apparently need not know/understand any fundamentals of optimal control theory to create/run codes for solving optimal control problems.  %
}

Because DIDO does not ask a user to supply the necessary conditions of optimality, it is frequently confused for a ``direct'' method.  Ironically, it is also confused with an ``indirect'' method because DIDO also provides costates and other multipliers associated with Pontryagin's Principle.
From the details provided in this paper, it is clear that DIDO does not fit within the traditional stovepipes of direct and indirect methods. DIDO achieves its true-fast computational speed without requiring a user to provide a ``guess,'' or analytical gradient/Jacobian/Hessian information or demanding a high-performance computer; rather, the ``formula'' behind DIDO's robustness, speed and accuracy is based on a Hamiltonian-centric approach to dynamic optimization whose fundamentals are agnostic to pseudospectral discretizations.  This is likely the reason why DIDO is able to solve a wide variety of ``hard'' differential-algebraic problems that other codes/methods have admitted difficulties.

\section*{Endnotes}
DIDO$^\copyright$ is a copyrighted code registered with the United States Copyright Office. Under Title 17 U.S.C. \S 107, fair use of DIDO's format and DIDO's structures is permitted (with proper citation of this article) but limited to research publications, news articles, scholarship, nonprofit or academic use.

%%%%%%%%
%Bibliography
%%%%%%%%

%\begin{spacing}{1.0}

%\end{spacing}

%\section*{Appendices}

%\appendix
%--------------------------
% APPENDIX
%--------------------------

\section*{Appendix A: DIDO's Computational Theory}\label{sec:DIDO-theorems}
A vexing question that has plagued computational methods for optimal controls is the treatment of differential constraints.  DIDO addresses this question by introducing a \textit{\textbf{virtual control variable}} $\bv$ to rewrite the differential constraint as,
\begin{subequations}\label{eq:vdef+f-eqns}
\begin{align}
\frac{d\bx}{d\tau} &:= \bv \label{eq:v=bydef-tau}\\
\left(\frac{d\Gamma}{d\tau}\right) \bff(\bx, \bu, \Gamma(\tau; t_0, t_f))  & = \bv \label{eq:f=v-in-tau}
\end{align}
\end{subequations}
where, $\Gamma: \Real \to \Real$ is an invertible nonlinear transformation\cite{ross-book,PSReview-ARC-2012,kt-map},
\begin{equation}
t = \Gamma(\tau; t_0, t_f) \quad \Leftrightarrow \tau = \Gamma^{-1}(t; t_0, t_f)
\end{equation}
and $\tau$ is transformed time over a fixed horizon $[\tau^0, \tau^f]$.  Because the virtual control variable affects only the dynamical constraints, the rest of the problem formulation remains unchanged. Thus the age-old question of the ``best'' way to discretize a nonlinear differential equation is relegated to answering the same question for the embarrassingly simple linear equation given by \eqref{eq:v=bydef-tau}.
\begin{remark}
The transformation given by \eqref{eq:vdef+f-eqns} is only internal to DIDO.  As far as the user is concerned a virtual control variable does not exist.
\end{remark}
As simple as it is, a surprisingly significant amount of sophistication is necessary to implement \eqref{eq:v=bydef-tau} efficiently. To explain the details of this computational theory, we employ the special problem defined by,
\begin{eqnarray}
& (\textsf{$P$}) \left\{
\begin{array}{lrl}
\emph{Minimize } & J[x(\cdot), u(\cdot)] :=&\displaystyle \int_{t_0}^{t_f} F(x(t), u(t))\,dt\\[1em]
\emph{Subject to}& \dot x(t) =& f(x(t), u(t))
\end{array} \right. & \label{eq:probP}
\end{eqnarray}
where, $t_0, t_f, x(t_0)$ and $x(t_f)$ are all fixed at some appropriate values.  Furthermore, because the time interval $[t_f - t_0]$ is fixed, there is no need to perform a ``$\tau$''-transformation to generate the virtual control variable.  Consequently, using \eqref{eq:vdef+f-eqns}, we construct the \textbf{\emph{DIDO-form}} of $(P)$ as,
\begin{eqnarray}
& (\textsf{$P^D$}) \left\{
\begin{array}{lll}
\emph{Minimize } & J[x(\cdot), u(\cdot), v(\cdot)] :=\\
& \qquad \displaystyle \int_{t_0}^{t_f} F(x(t), u(t))\,dt\\[1em]
\emph{Subject to} & \dot x(t) =  v(t)\\
& 0 =  f(x(t), u(t)) - v(t)
\end{array} \right. & \label{eq:probPD}
\end{eqnarray}
\subsection{Introduction to the Cotangent Barrier}
Using the equations and the process described in Sec.~\ref{sec:DIDO-i/o} in applying Pontraygin's Principle to Problem~$(P)$, it is straightforward to show that the totality of conditions results in the following nonlinear differential-algebraic equations:
\begin{eqnarray}
& (\textsf{$P^\lambda$}) \left\{
\begin{array}{lrl}
& \dot x(t) =& f(x(t), u(t))  \\
&-\dot\lambda(t) = & \partial_x H(\lambda(t), x(t), u(t))\\
& 0 = &\partial_u H(\lambda(t), x(t), u(t))
\end{array} \right. & \label{eq:probPlam}
\end{eqnarray}
Using the concept of the virtual control variable for both the state and the costate equations, the \emph{DIDO-form} of $(P^\lambda)$ is given by,
\begin{eqnarray}
& (\textsf{$P^{\lambda, D}$}) \left\{
\begin{array}{lrl}
& \dot x(t) =& v(t)  \\
&-\dot\lambda(t) = & \omega(t)\\
& 0 =& f(x(t), u(t)) - v(t)  \\
&0 = & \partial_x H(\lambda(t), x(t), u(t)) - \omega(t)\\
& 0 = &\partial_u H(\lambda(t), x(t), u(t))
\end{array} \right. & \label{eq:probPlamD}
\end{eqnarray}
%------------------------
where, $\omega$ is called the co-virtual control variable. DIDO is based on the principle that \eqref{eq:probPlamD} is computationally easier than \eqref{eq:probPlam} to the extent that all the differential components are concentrated in the exceedingly simple linear equations.  To appreciate the simple sophistication of this principle, consider now an application of Pontryagin's Principle to Problem~$(P^D)$; this  generates,
\begin{eqnarray}
& (\textsf{$P^{D,\lambda}$}) \left\{
\begin{array}{lrl}
& \dot x(t) =& v(t)  \\
& 0 = & f(x(t), u(t))-v(t)  \\
&-\dot\lambda(t) = & \partial_x H(\mu(t), x(t), u(t))\\
& \lambda(t) = & \mu(t)\\
& 0 = &\partial_u H(\mu(t), x(t), u(t))
\end{array} \right. & \label{eq:probPDlam}
\end{eqnarray}
%------------------------
Comparing \eqref{eq:probPDlam} to \eqref{eq:probPlamD}, it is clear that,
%--------------
\begin{equation}\label{eq:cotan-barrier}
(P^{D, \lambda}) \not\equiv (P^{\lambda, D})
\end{equation}
%-----------------
That is, \eqref{eq:cotan-barrier} is a statement of the noncommutativity  of the ``$\lambda$'' and ``$D$'' operations.  This noncommutativity can be written more elaborately as,
%-----------------
\begin{equation}\label{eq:cotan-barrier-operator-form}
\big(\lambda \circ D \big) (P) \not\equiv  \big( D\circ\lambda \big) (P)
\end{equation}
%------------------
Figure~\ref{fig:cotangentB} encapsulates this notion in terms of an ``insurmountable'' \textbf{\emph{cotangent barrier}}.
%
%======================================================================================
   \begin{figure}[h!]
      \centering
      \framebox{\parbox{0.9\columnwidth}{
      \centering
      {\includegraphics[angle=-0, width = 0.63\columnwidth]{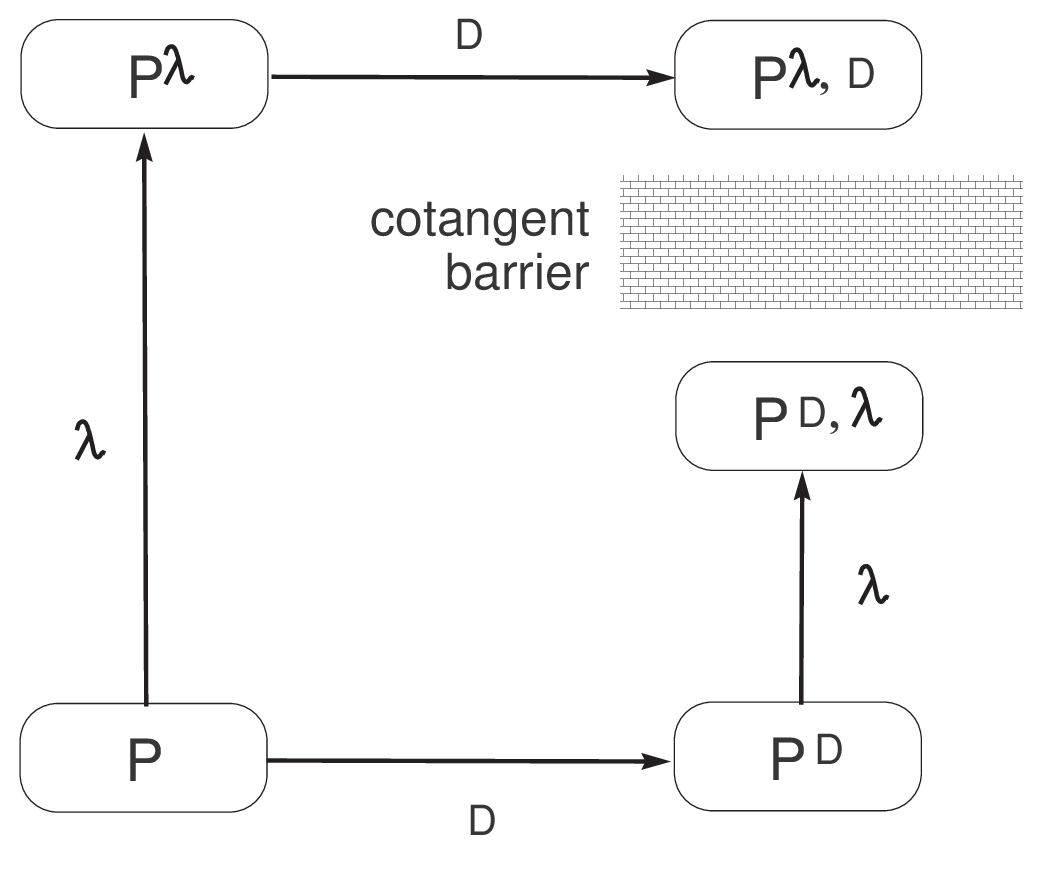}}
      \caption{\textsf{An illustration of the cotangent barrier indicating the impossibility of moving ``up'' (i.e. in $\blam$) from $(P^{D, \lambda})$ to $(P^{\lambda, D})$.}}
    \label{fig:cotangentB}
    }}
   \end{figure}
%==========================================================================================
%
It depicts the fact that an application of Pontryagin's Principle to the DIDO-form of the primal problem does not generate a DIDO-form of the primal-dual problem. The noncommutativity of the two operations seems to suggest that there was no value in using the DIDO form of the primal problem.  Despite this apparently discouraging result, we now show that it is possible to tunnel through the barrier, albeit discretely (pun intended!).

\subsection{Discrete Tunneling  Through the Cotangent Barrier}
Let $X^N, U^N$, and $V^N$ be $(N+1)$-dimensional vectors that represent discretized state, control and virtual control variables respectively over an \textbf{\emph{arbitrary grid}} $\pi^N$:
%---------
\begin{subequations}\label{eq:XUVNdef}
\begin{align}
X^N  &:= (x_0, x_1, \ldots, x_N)\\
U^N  &:= (u_0, u_1, \ldots, u_N)\\
 V^N &:= (v_0, v_1, \ldots, v_N)
\end{align}
\end{subequations}
From linearity arguments, it follows that a broad class of generic discretizations of $\dot x = v$ may be written as,
\begin{equation}\label{eq:AXVN=0}
\bA \left(
    \begin{array}{c}
      X^N \\
      V^N \\
    \end{array}
  \right) \equiv \bA_x X^N + \bA_v V^N = C^N
\end{equation}
where, $\bA  := [\bA_x \mid \bA_v] $ comprises two $(N+1) \times (N+1)$ discretization matrices that depend on $N$, $\pi^N$, and a choice of the discretization method. The choice of the discretization method also determines $C^N$, an $(N+1) \times 1$  matrix.  The dependence of $\bA_x$ and $\bA_v$ on $N$ and $\pi^N$ is suppressed for notational convenience.
Following the same rationale as in the production of \eqref{eq:AXVN=0},  a generic discretization of $\dot\lambda = -\omega$ may be written as,
\begin{equation}\label{eq:Astar-def}
\bA^* \left(
    \begin{array}{c}
      \Lambda^N \\
      \Omega^N \\
    \end{array}
  \right) \equiv \bA^*_\lambda \Lambda^N + \bA^*_\omega \Omega^N = C^{*N}
\end{equation}
where $\Lambda^N \in \real{N+1}$ and $\Omega^N \in \real{N+1}$ are defined analogous to \eqref{eq:XUVNdef}, $\bA^*$ is a matrix (that may not necessarily be equal to $\bA$) that depends on $N$ and $\pi^N$, and $C^{*N}$ is a matrix similar to $C^N$.
Collecting all relevant equations, it follows that a generic discretization of $(P^{\lambda, D})$ may be written as,
\begin{eqnarray}
& (\textsf{$P^{\lambda, D, N}$}) \left\{
\begin{array}{lrl}
& \bA \left(
    \begin{array}{c}
      X^N \\
      V^N \\
    \end{array}
  \right) =& C^N\\[1em]
&\bA^* \left(
    \begin{array}{c}
      \Lambda^N \\
      \Omega^N \\
    \end{array}
  \right) = & C^{*N} \\ [1em]
& \partial_{\lambda}H(\Lambda^N, X^N, U^N) - V^N = & 0 \\[0.5em]
&\partial_{x}H(\Lambda^N, X^N, U^N) - \Omega^N = &0 \\[0.5em]
&\partial_{u}H(\Lambda^N, X^N, U^N)  = &0
\end{array} \right. &\label{eq:PlamDN}
\end{eqnarray}
%-------------
where, $H$ is reused as an overloaded operator to take in discretized vectors as inputs.

DIDO is based on the ansatz that \eqref{eq:PlamDN} is a well-structured, \emph{\textbf{grid-agnostic}} root-finding problem that can be solved efficiently for any given $N$ and $\pi^N$.  The basis for this ansatz is as follows:
\begin{enumerate}
\item The $\bA$-$\bA^*$ system is agnostic to the problem definition and data functions; hence, it can be handled separately and ``independently'' over the ``space of all optimal control problems.''
\item The $\bA$-$\bA^*$ system is always a linear (affine) equation even if Problem $(P)$ contains inequalities.
\item The nonlinear elements of \eqref{eq:PlamDN} are based on the Hamiltonian of Problem $(P)$.  The Hamiltonian is constructed by a simple dot-product operation (see Sec.~\ref{sec:DIDO-i/o}). In its discretized (overloaded) form, $H$ has the property that the Hamiltonian and its gradients at $i$ are not connected to their values at $(i+1)$.  This separable programming\cite{GMSW} property of the Hamiltonian and its gradients over $\pi^N$ can be harnessed for computational efficiency.
\item Solving the nonlinear elements of \eqref{eq:PlamDN} involves a sequential linearization as an inner-loop. Augmenting the $\bA$-$\bA^*$ linear system to this inner loop solves the entire problem. The iterations for solving the augmented linear system can be tailored to the special structure of the sequentially linearized Hamiltonian system and the specific method of discretization captured in the $\bA$-$\bA^*$ system.
\end{enumerate}
%
%Finally, the iterations for solving $(P^{\lambda, D, N})$ for $N = N_2 > N_1$ can be substantially reduced through the use of computationally efficient projection operators; see \cite{spec-alg}.
All these features of  $(P^{\lambda, D, N})$ are exploited in the production of DIDO's algorithm as detailed in Appendix~B.
%
%---------------------
\begin{definition}
\textbf{DIDO's generalized equation} is defined by $(P^{\lambda, D, N})$.
\end{definition}
%------------------
%
Similar to $(P^{\lambda, D})$, Problem~$(P^D)$ may be discretized to generate the following problem,
\begin{eqnarray}
& (\textsf{$P^{D,N}$}) \left\{
\begin{array}{lll}
\emph{Minimize } & J^{D,N}[X^N, U^N, V^N] \\[0.3em]
                 & \qquad := \bq^T F(X^N, U^N)\\[0.6em]
\emph{Subject to}& \bA \left(
    \begin{array}{c}
      X^N \\
      V^N \\
    \end{array}
  \right) = C^N\\[1em]
& f(X^N, U^N) - V^N =  0
\end{array} \right. & \label{eq:probPDN}
\end{eqnarray}
where, $\bq = (q_0, q_1, \ldots, q_N)$ is an $(N+1) \times 1$ vector of positive quadrature weights associated with the specifics of the discretization given by $\bA$ (and inclusive of $\pi^N$).
%
%-------------------
\begin{remark}
From \eqref{eq:PlamDN} and \eqref{eq:probPDN}, it follows that DIDO's entire approach to discretization is isolated in the matrix pair $\set{\bA, \bA^*}$.  Part of DIDO's computational efficiency is achieved by exploiting this linear system.
\end{remark}
%--------------------
%===============
\begin{lemma}\label{lemma:pre-CMT4D}
Let $\bQ$ be the positive definite diagonal matrix defined by $\text{diag }(q_0, q_1, \ldots, q_N)$.  Define,
\begin{equation}\label{eq:Adag=bydef}
\bA^\dag_x := \bQ^{-1}\bA_x^T \bQ \quad\text{and}\quad  \bA^\dag_v := \bQ^{-1}\bA_v^T\bQ
\end{equation}
Then, under appropriate technical conditions at the boundary points, there exist multipliers $\Psi^N_A \in \real{N+1}$ and $\Psi^N_d \in \real{N+1}$ for Problem~$(P^{D, N})$ such that its dual feasibility conditions can be written as,
\begin{align}
%\bA_x X^N + \bA_v V^N & = C^N\\
%f(X^N, U^N) - V^N & = 0 \\
\bA^\dag_v \overline{\Psi}^N_A - \overline{\Psi}^N_d  & = C^{*N} \label{eq:Hope4Astar}\\
\partial_x H\left(\overline{\Psi}^N_d, X^N, U^N\right)  + \bA^\dag_x \overline{\Psi}^N_A  &=0\\
\partial_u H\left(\overline{\Psi}^N_d, X^N, U^N\right)  &=0
\end{align}
where $\overline{\Psi}^N_A$ and $\overline{\Psi}^N_d$ are given by,
\begin{equation}
\overline{\Psi}^N_A := \bQ^{-1} \Psi^N_A \quad  \overline{\Psi}^N_d := \bQ^{-1} \Psi^N_d
\end{equation}
\end{lemma}
%---------------
\begin{proof}
This lemma can be proved by an application of the multiplier theory.  Details are omitted for brevity.
\end{proof}
%===============
\\

It follows from Lemma~\ref{lemma:pre-CMT4D} that if $\bA^\dag$ (and hence $\bA$) can be chosen such that its $x$- and $v$-components constitute $\bA^*$ of \eqref{eq:PlamDN}, then, it will be possible to commute the noncommutative operations inherent in the cotangent barrier, except that this would have been achieved in the discretized space.  This idea is captured in Fig.~\ref{fig:cotangentT} and the statement of its possibility is enunciated as the \textbf{\emph{tunnel theorem}}.
%
%
%======================================================================================
   \begin{figure}[h!]
      \centering
      \framebox{\parbox{0.9\columnwidth}{
      \centering
      {\includegraphics[angle=-0, width = 0.9\columnwidth]{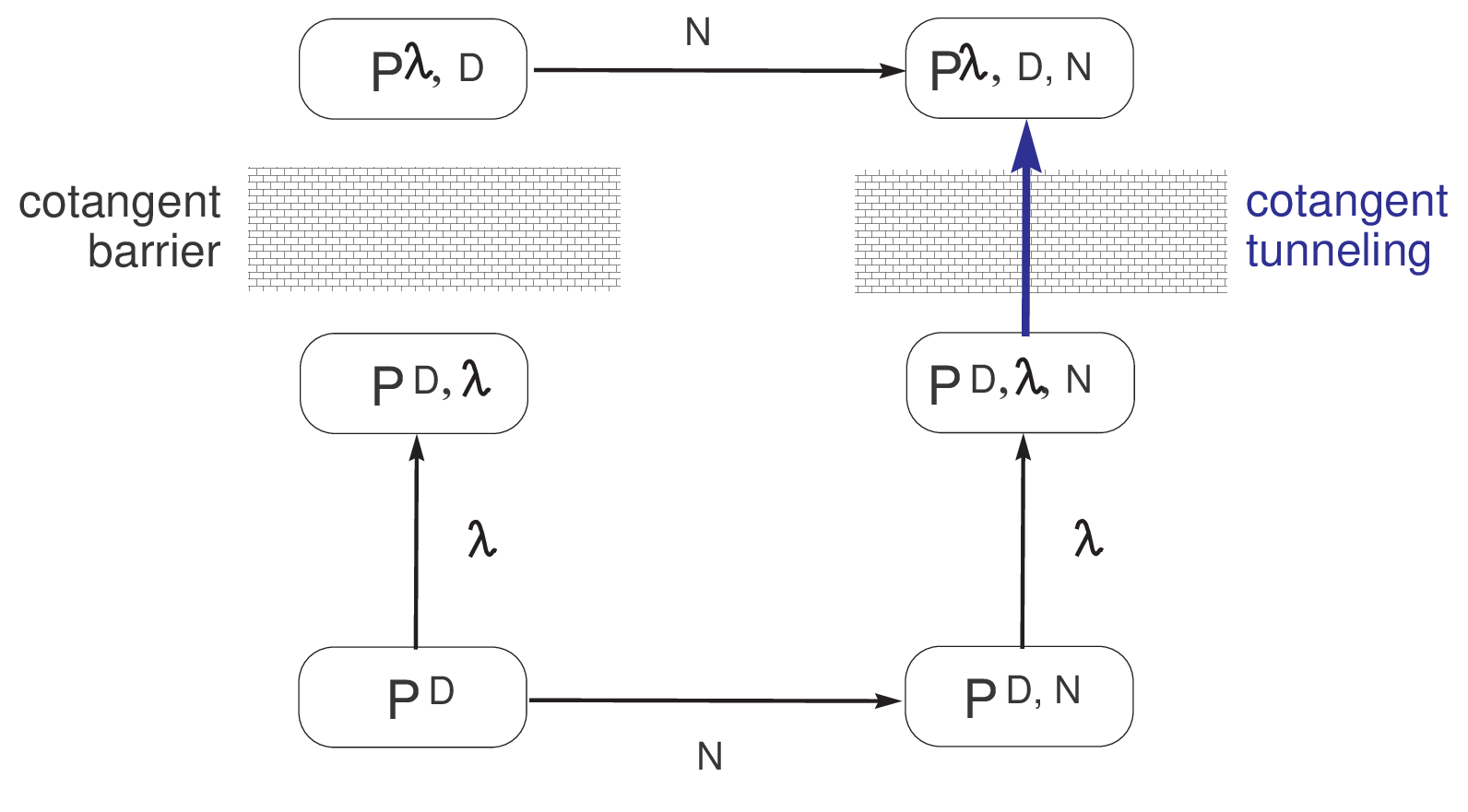}}
      \caption{\textsf{Discrete tunneling allows $(P^{D, \lambda, N})$ to move through the cotangent barrier to connect with $(P^{\lambda, D, N})$.}}
    \label{fig:cotangentT}
    }}
   \end{figure}
%==========================================================================================
%
%==================
\begin{theorem}[A Tunnel Theorem]\label{theorem:tunnel}
Suppose $\bA$ and $\bA^*$ are discretization matrix pairs that satisfy the conditions,
\begin{subequations}\label{eq:tunnel-conditions}
\begin{align}
\bA_x  &= -\boldsymbol{I}^N \\
\bA^*_\lambda &= -\boldsymbol{I}^N \\
\bA^*_\omega  &= \bA^\dag_v
\end{align}
\end{subequations}
then, the necessary conditions for Problem $(P^{D,N})$ are equivalent to $(P^{\lambda,D,N})$ under the transformation,
\begin{subequations}\label{eq:tunnel-transforms}
\begin{align}
\overline{\Psi}^N_d &= \Lambda^N \\
\overline{\Psi}^N_A  &= \Omega^N
\end{align}
\end{subequations}
\end{theorem}
%---------------------
\begin{proof}
The proof of this theorem follows from Lemma~\ref{lemma:pre-CMT4D} and by a substitution of \eqref{eq:tunnel-conditions} and \eqref{eq:tunnel-transforms} in the the definitions of $(P^{D,N})$ and $(P^{\lambda,D,N})$.
\end{proof}
%=====================

%-----------------------------
\begin{definition}\label{def:DIDO-HPP}
$(P^{D,N})$ is called a \textbf{DIDO-Hamiltonian programming problem} if $\bA$ and $\bA^*$ are chosen in accordance with Theorem~\ref{theorem:tunnel}.
\end{definition}
%==================================
See \cite{ross:direct-shooting} for a first-principles introduction to the notion of \emph{\textbf{Hamiltonian programming}}.  A DIDO-Hamiltonian programming problem is simply an adaptation of this terminology to Problem~$(P^{D,N})$.
%=============================
\begin{theorem}\label{theorem:Lag=not}
All Lagrange PS discretizations over any grid (including Gauss, Radau, and Lobatto) fail to satisfy the conditions of Theorem~\ref{theorem:tunnel}.
\end{theorem}
%----------------------------
\begin{proof}
The proof of this theorem is fairly straightforward.  It follows from the fact that all Lagrange PS discretizations\cite{wang, birk-koeppen,PSReview-ARC-2012} over any grid\cite{advances, arb-grid-aas,arb-grid} are based on a differentiation matrix, $\bD$.  Hence all Lagrange PS discretizations of $\dot x = v$ are given by,
\begin{equation}
\bD X^N = V^N \quad \Rightarrow \quad V^N - \bD X^N = \bzero
\end{equation}
This implies $\bA_x = -\bD$ and $\bA_v = \boldsymbol{I}^N$ which violates the conditions of Theorem~\ref{theorem:tunnel}.
\end{proof}
%===============================
\begin{theorem}\label{theorem:Birk=tunnel}
A collection of Birkhoff PS discretizations of optimal control problems satisfy the conditions of Theorem~\ref{theorem:tunnel}.
\end{theorem}
%---------------------------
\begin{proof}
The proof of this theorem follows from the fact that a Birkhoff PS discretization\cite{birk-koeppen,birk-TN} may be written in a form that generates $\bA_x = - \boldsymbol{I}^N$; see \cite{birk-TN} for further details.
\end{proof}
%===========================
\\

\noindent
Theorems \ref{theorem:Lag=not} and \ref{theorem:Birk=tunnel} imply that there are two broad categories of PS discretizations for optimal control problems, namely Lagrange and Birkhoff. Within these two main methods of discretizations, it is possible to generate a very large number of variations based on the choice of basis functions\footnote{The basis functions need not be polynomials\cite{birk-TN,atap,stenger:sinc}. Nonpolynomial ``designer'' basis functions can also be generated via the nonlinear domain transformation $\Gamma$ used in \eqref{eq:f=v-in-tau}; see \cite{ross-book,atap,PSReview-ARC-2012,kt-map} for details. } and grid selections.  \emph{\textbf{A process to achieve at least eighteen variations of PS discretizations based on classical orthogonal polynomials is shown in Fig.~\ref{fig:PStypes}}}.
%
%======================================================================================
   \begin{figure}[h!]
      \centering
      \framebox{\parbox{0.9\columnwidth}{
      \centering
      {\includegraphics[angle=0, width = 0.9\columnwidth]{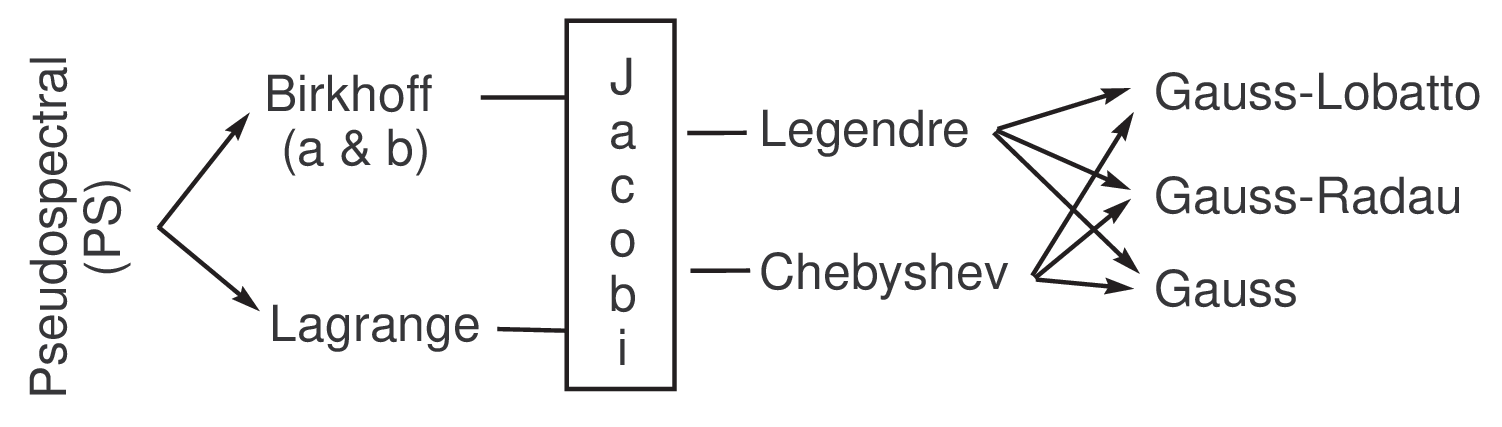}}
      \caption{\textsf{Schematic for generating at least 18 variants of PS discretizations; figure adapted from \cite{birk-TN}.}}
    \label{fig:PStypes}
    }}
   \end{figure}
%==========================================================================================
%
%All of the Lagrange variants (including a switch between variants) have been implemented in the older versions of DIDO (see Sec.~\ref{sec:misc}).
Despite this apparently large variety of choices, Theorems \ref{theorem:Lag=not} and \ref{theorem:Birk=tunnel} reveal that there are no essential differences between PS discretizations based on different grid points. All that matters for \emph{optimal control applications} is whether they are based on Lagrange or Birkhoff interpolants.  Nonetheless, when additional technical factors are taken into account, such as the inherent structure of the resulting inner-product space\cite{advances}, the choice of a grid can have a deleterious effect on convergence; see \cite{PSReview-ARC-2012,advances} for details.  See also \cite{ross-book}, Sec.~4.4.4.

\section*{Appendix B: DIDO's Suite of Algorithms}\label{sec:algol-suite}
As noted in Sec.~\ref{sec:intro}, DIDO's algorithm is actually a suite of algorithms. Several portions of this suite can be found in various publications scattered across different journals and conference proceedings. In this section, we collect and categorize this algorithm-suite to explain how specific algorithmic options are triggered by the particulars of the user-defined problem and inputs to DIDO.

\subsection{Introduction and Definitions}

Each main algorithm in DIDO's suite is \emph{\textbf{true-fast, spectral and Hamiltonian}}. That is, each of DIDO's main algorithm solves the DIDO-Hamiltonian programming problem (see Definition~\ref{def:DIDO-HPP}) using a modified\cite{birk-koeppen,birk-TN} spectral algorithm\cite{spec-alg,arb-grid} that is true-fast.  True-fast is different from computationally fast in the sense that the former must be agnostic to the details of its computer implementations whereas the latter can be can be achieved via a variety of simple and obvious ways.  For instance, any optimization software, including a mediocre one, can be made computationally fast by using one or more of the following:
\begin{enumerate}
\item A good/excellent ``guess;''
\item A high-performance computer;
\item A compiled or embedded code; and,
\item Analytic gradient/Jacobian/Hessian information.
\end{enumerate}
A truly fast algorithm should be \emph{independent} of these obvious and other run-time code improvements.  \emph{\textbf{For an algorithm to be true-fast, it must satisfy the following properties}} (see Fig.~\ref{fig:algorithmDots}):
\begin{enumerate}
\item Converge from an arbitrary starting point;
\item Converge to an optimal/extremal solution;\footnote{An extremal solution is one that satisfies Pontryagin's Principle.}
\item Take the fewest number of iterations towards a given tolerance criterion; and/or
\item Use the fewest number of operations per iteration.
\end{enumerate}
%
%The DIDO optimal control toolbox is based on a suite of true-fast algorithms that work in unison to advance the three key performance elements introduced in Sec.~I. To better explain the me
%
%Automatic triggers are initiated based on the specifics of the problem and the user's inputs.
%
%
%======================================================================================
   \begin{figure}[h!]
      \centering
      \framebox{\parbox{0.9\columnwidth}{
      \centering
      {\includegraphics[angle=0, width = 0.9\columnwidth]{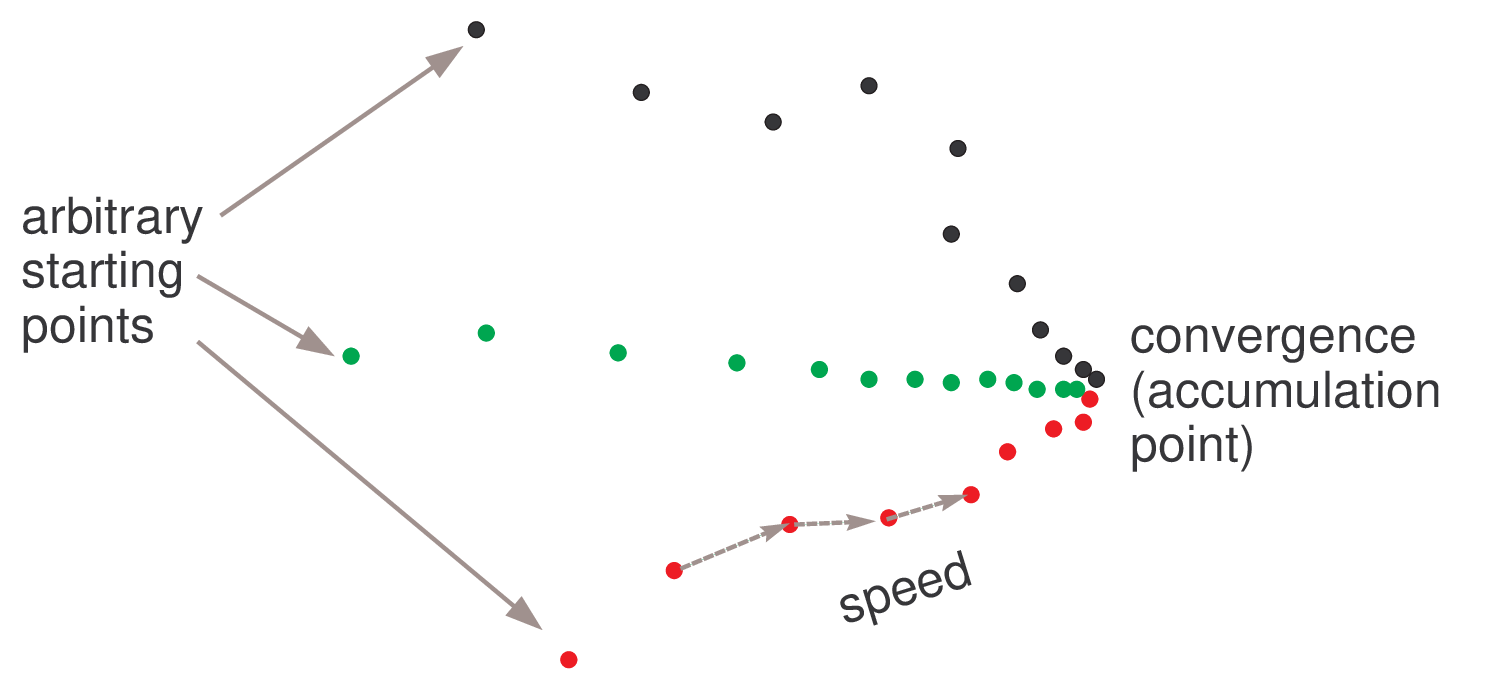}}
      \caption{\textsf{Schematic for defining an algorithm and its fundamental properties.}}
    \label{fig:algorithmDots}
    }}
   \end{figure}
%==========================================================================================
%
Thus, a true-fast algorithm can be easily made computationally fast, but its implementation may not necessarily be computationally fast.  Likewise, a purportedly computationally fast algorithm may not necessarily be true-fast. DIDO achieves true-fast speed by calling several true-fast algorithms from its toolbox.  The specifics of a given problem and user inputs initiate automatic triggers to generate a coordinated DIDO-iteration that defines the flow of the main algorithm. To describe this algorithm suite, we use with the following definitions from \cite{ross:direct-shooting}:
\begin{definition}[Inner Algorithm $\mathcal{A}$]\label{def:inner-algol}
Let $(P^N)$ denote any discretization of $(P)$. An inner algorithm $\mathcal{A}$ for Problem~$(P)$ is defined as finding a sequence of vector pairs $(X^N_0, U^N_0), (X^N_1, U^N_1), \ldots $
by the iterative map,
$$(X^N_{i+1}, U^N_{i+1}) = \mathcal{A}(X^N_i, U^N_i), \ i \in \mathbb{N}$$
\end{definition}
%------------------------
Note that $N$ is fixed in the definition of the inner algorithm.
%-----------------------
\begin{definition}[Convergence of an Inner Algorithm]\label{def:inner-conv}
An inner algorithm $\mathcal{A}$ is said to converge if
$$\lim_{n \to \infty} X^N_n := X^N_\infty , \qquad \lim_{n \to \infty} U^N_n  := U^N_\infty $$
is an accumulation point that solves Problem $(P^N)$.
\end{definition}
%------------------------
\begin{definition}[Hamiltonian Algorithm $\mathcal{A}^\lambda$ ]
Let $(P^{\lambda, N})$ denote any discretization of $(P^\lambda)$.  A (convergent) inner algorithm $\mathcal{A}^\lambda$ is said to be Hamiltonian if it generates an additional sequence of vectors $\Lambda^N_0, \Lambda^N_1,  \ldots $ for a fixed $N$ such that
$$\lim_{n \to \infty} \Lambda^N_n := \Lambda^N_\infty$$
is an accumulation point that solves $(P^{\lambda, N})$.
\end{definition}
%------------------------
\begin{remark}
Note that the Hamiltonian algorithm is focused on solving $(P^{\lambda, N})$ not $(P^{N, \lambda})$.  In general, $(P^{\lambda, N}) \ne (P^{N, \lambda})$ for finite $N$ even if they are equal in the limit as $N \to \infty$.
\end{remark}
%=========================================
\begin{lemma}\label{lemma:ross-funda-1}
Suppose $(P^{\lambda, N}) = (P^{N, \lambda})$ for finite $N$.   If $(P^{N, \lambda})$ is solved to a practical tolerance of $\delta^{N, \lambda} > 0$, then, the resulting solution solves $(P^{\lambda, N})$ to a tolerance of  $\delta^{\lambda, N} > 0$ such that
\begin{equation}\label{eq:ross-funda-lemma-ineq}
\delta^{\lambda, N} > \delta^{N, \lambda}
\end{equation}
\end{lemma}
%-----------------------
\begin{proof}
This lemma is proved in \cite{ross:direct-shooting} for a shooting method and Euler discretization. Its extension to the general case as given by \eqref{eq:ross-funda-lemma-ineq} follows from the assumption that $(P^{\lambda, N}) = (P^{N, \lambda})$ for finite $N$.
\end{proof}
%==============================
\\

\noindent
From \eqref{eq:ross-funda-lemma-ineq} it follows that $\delta^{\lambda, N} \ne \delta^{N, \lambda}$.  A Hamiltonian algorithm fixes this discrepancy.

As explained in \cite{ross:direct-shooting}, Problem~$(P^N)$ is better categorized as a Hamiltonian programming problem rather than as a nonlinear programming problem.  This is, in part, because \textbf{\emph{a Hamiltonian is not a Lagrangian}} and Problem~$(P^N)$ contains ``hidden'' information that is not accessible via $(P^{N,\lambda})$. A Hamiltonian algorithm aims to solve $(P^{\lambda, N})$ while a nonlinear programming algorithm attempts to solve $(P^{N, \lambda})$.  Even when these two problems are theoretically equivalent to one another, a nonlinear programming algorithm does not solve $(P^{\lambda, N})$ to the same accuracy as $(P^{N, \lambda})$ as shown in Lemma~\ref{lemma:ross-funda-1}; hence, a Hamiltonian algorithm is needed to replace or augment generic nonlinear programming algorithms. These ideas are depicted in Fig.~\ref{fig:HPPvenn}.
%
%======================================================================================
   \begin{figure}[h!]
      \centering
      \framebox{\parbox{0.9\columnwidth}{
      \centering
      {\includegraphics[angle=0, width = 0.9\columnwidth]{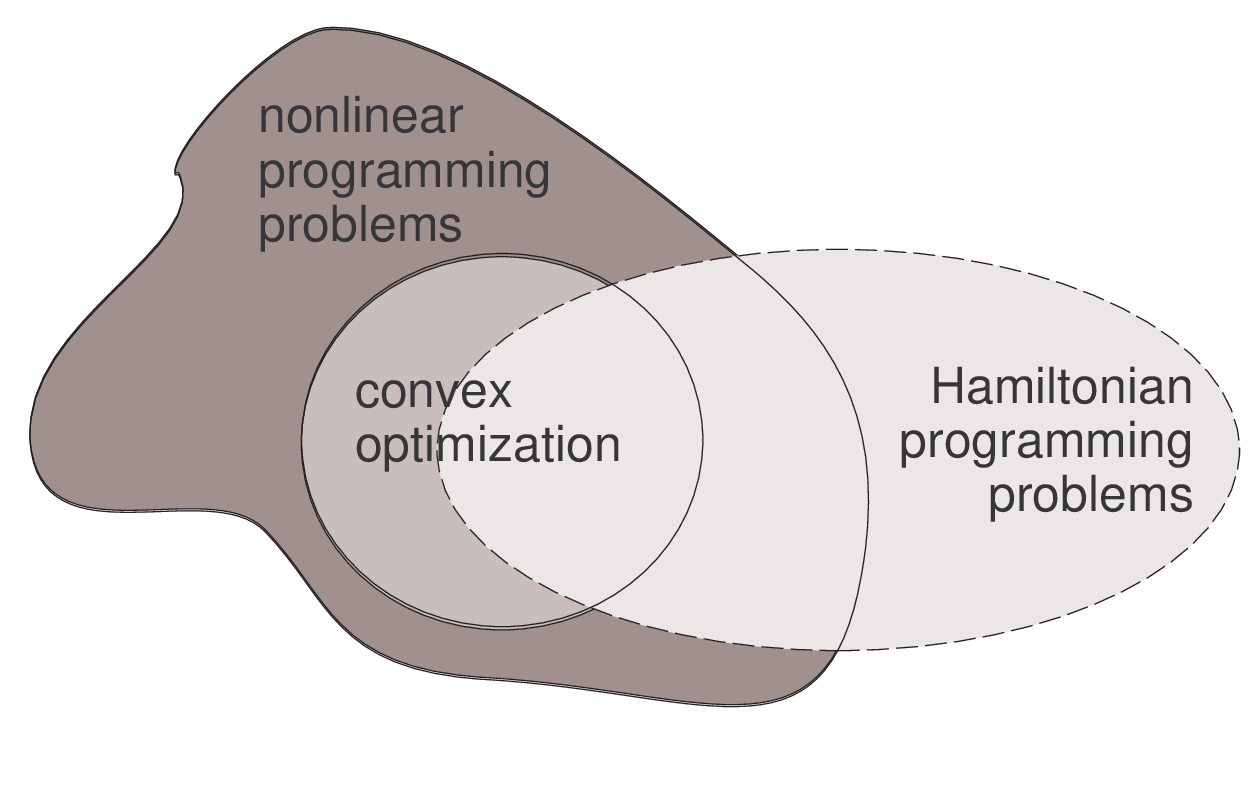}}
      \caption{\textsf{A depiction of the differences between nonlinear and Hamiltonian programming.}}
    \label{fig:HPPvenn}
    }}
   \end{figure}
%==========================================================================================
%
Specific details in augmenting a nonlinear programming algorithm with a Hamiltonian component may be found in \cite{biegler2014,biegler2016,hamilMeshRef}.  Alternative implementations that achieve similar effects are discussed in \cite{acc:hybrid,arb-grid-aas,arb-grid,cheby:+cmt}.  Note also that in Hamiltonian programming, the separable-programming-type property inherent in $(P^{D, N})$ is naturally exploited.  In the multi-variable case, this property lends itself to casting the discretized problem in terms of a compact matrix-vector programming problem rather than a generic a nonlinear programming problem\cite{birk-TN}.

%---------------------------
\begin{definition}[DIDO-Hamiltonian Algorithm $\mathcal{A}^{\lambda, D}$]\label{def:DIDO-c-inner}
A Hamiltonian algorithm $\mathcal{A}^{\lambda, D}$ is said to be DIDO-Hamiltonian if it generates an additional sequence of vector pairs $(V^N_0 \Omega^N_0), (V^N_1, \Omega^N_1),  \ldots $ for a fixed $N$ such that
$$\lim_{n \to \infty} V^N_n := V^N_\infty, \quad \lim_{n \to \infty} \Omega^N_n := \Omega^N_\infty$$
is an accumulation point that solves Problem $(P^{\lambda, D, N})$.
\end{definition}
%----------------------
\begin{remark}\label{remark:lemma-14c}
Because $(P^{\lambda, D}) \not\equiv  (P^{D, \lambda})$  (see \eqref{eq:cotan-barrier}), we have $(P^{\lambda, D, N}) \not\equiv  (P^{D, \lambda, N})$ unless the discretization offers a cotangent tunnel; see Theorem~\ref{theorem:tunnel}.  Even if $(P^{\lambda, D, N}) \equiv  (P^{D, \lambda, N})$, it follows from Lemma~\ref{lemma:ross-funda-1} that $\delta^{\lambda, D, N} \ne  \delta^{D, \lambda, N}$, where $\delta^{\lambda, D, N}$ and  $\delta^{D, \lambda, N}$ denote the constraint tolerances on $(P^{\lambda, D, N})$ and $(P^{D, \lambda, N})$ respectively.  The DIDO-Hamiltonian algorithm fixes this discrepancy.
\end{remark}
%----------------------
A solution to Problem~$(P^N)$ does not imply a solution to Problem~$(P)$.  In fact, an optimal solution to Problem~$(P^N)$ may not even be a feasible solution to Problem~$(P)$. This point is often lost when nonlinear programming solvers are patched to solve discretized optimal control problems.  To clarify this point, we use the following definitions from \cite{ross:direct-shooting}:
%----------------------
\begin{definition}[Convergence of a Discretization]\label{def:disc-conv}
Let $(X^N_\infty, U^N_\infty)$ be a solution to Problem~$(P^N)$.  A discretization is said to converge if
$$\lim_{N \to \infty} X^N_\infty := X^\infty_\infty , \qquad \lim_{N \to \infty} U^N_\infty := U^\infty_\infty $$
is an accumulation point that solves Problem $(P)$.
\end{definition}
%----------------------------
\begin{remark}
In Definition~\ref{def:disc-conv} and subsequent ones to follow, we have taken some mathematical liberties with respect to the precise notion of a metric for convergence.  It is straightforward to add such precision; however, it comes at a significant cost of new mathematical machinery.  Consequently, we choose not to include such additional mathematical nomenclature in order to support the accessibility of the proposed concepts to a broader audience.
\end{remark}
%---------------------------
\begin{definition}[Dynamic Optimization Algorithm $\mathcal{B}$]\label{def:dyn-opt}
Let $\mathcal{B}$ be an algorithm that generates a sequence of integer pairs $(N_0, m_0),$ $(N_1, m_1) \cdots $ such that the sequence
$$(X^{N_0}_{m_0}, U^{N_0}_{m_0}), (X^{N_1}_{m_1}, U^{N_1}_{m_1}), \ldots  $$
converges to an accumulation point $(X^\infty_\infty, U^\infty_\infty)$ given by Definition~\ref{def:disc-conv}. Then, $\mathcal{B}$ is called a dynamic optimization algorithm for solving Problem $(P)$.
\end{definition}
%---------------------------------------------------------
%
From Definitions \ref{def:disc-conv} and \ref{def:dyn-opt}, it follows that a generic dynamic optimization algorithm involves the production of a double infinite sequence of vectors\cite{ross:direct-shooting}.  From Definition~\ref{def:disc-conv} it follows that the production of this double infinite sequence may be greatly facilitated by the use of an inner algorithm; see Definition~\ref{def:inner-algol}. From Definition~\ref{def:dyn-opt}, it follows that if an inner algorithm is used to construct $\mathcal{B}$, then it need not even be executed to completion! This idea is part of DIDO's process for generating a true-fast optimization algorithm.

Although the idea of not taking the inner algorithm to completion was motivated by Definition~\ref{def:dyn-opt}, the definition itself is driven by the theory of consistent approximations\cite{spec-alg,Kang_2008_convergence, TAC:linearizable,polak:book,CMT:CDC06}.  In fact, $m_k < \infty$ for $k < \infty$ is actually a theoretical requirement for consistency. Consistency requires that $\delta^{\lambda, N} > \epsilon^N > 0$ (see Lemma \ref{lemma:ross-funda-1}) where $\epsilon^N$ must be sufficiently large for finite $N$ and must have the property $\epsilon^N \to 0$ as $N \to \infty$. The rate at which $\epsilon^N$ goes to zero cannot be faster than the order of convergence; for example, $\epsilon^N \to 0$ cannot be faster than linear for Euler approximations.  For PS approximations, we may choose $\epsilon^N \to 0$ as $N \to \infty$ at the spectral or near-exponential rates.

%---------------------------------
\begin{definition}[Hamiltonian Dynamic Optimization Algorithm $\mathcal{B}^\lambda$]\label{def:DIDO-old-dyn-opt}
A dynamic optimization algorithm $\mathcal{B}^\lambda$ is said to be Hamiltonian if it generates an additional sequence of vectors
$$\Lambda^{N_0}_{m_0}, \Lambda^{N_1}_{m_1},  \ldots  $$
such that they converge to an accumulation point $(X^\infty_\infty, U^\infty_\infty, \Lambda^\infty_\infty)$ that solves $(P^\lambda)$.

\end{definition}

%---------------------------------------------------------
%
\begin{definition}[DIDO-Hamiltonian Dynamic Optimization Algorithm]\label{def:DIDO-dyn-opt}
A Hamiltonian dynamic optimization algorithm $\mathcal{B}^{\lambda, D}$ is said to be  DIDO-Hamiltonian if it generates an additional sequence of vector pairs
$$(V^{N_0}_{m_0}, \Omega^{N_0}_{m_0}), (V^{N_0}_{m_0}, \Omega^{N_1}_{m_1}),  \ldots  $$
that converges to an accumulation point $(X^\infty_\infty, U^\infty_\infty, V^\infty_\infty, \Lambda^\infty_\infty, \Omega^\infty_\infty)$ that solves $(P^{\lambda, D})$.

\end{definition}
%====================================
%

It is clear that Definition~\ref{def:DIDO-dyn-opt} relies on Definitions~\ref{def:DIDO-HPP}--\ref{def:DIDO-old-dyn-opt}. From these definitions it follows that DIDO-Hamiltonian dynamic optimization algorithm generates a double-infinite sequence of tuples,
\begin{multline}\label{eq:dido-c-sequence}
(X, U, V, \Lambda, \Omega)^{N_0}_{m_0}, (X, U, V, \Lambda, \Omega)^{N_1}_{m_1}, \ldots, (X, U, V, \Lambda, \Omega)^{N_f}_{m_f},\\
\ldots, (X, U, V, \Lambda, \Omega)^{\infty}_{\infty}
\end{multline}
whose limit point purportedly converges to a solution of $P^{\lambda, D}$.

\subsection{Overview of The Three Major Algorithmic Components of DIDO}

The production of each element of the sequence denoted in \eqref{eq:dido-c-sequence} is determined by suite of algorithms that call upon different inner-loop DIDO-Hamiltonian algorithms (see Definition~\ref{def:DIDO-c-inner}).  The specific calls on the inner loop are based upon the performance of the outer loop (see Definition~\ref{def:DIDO-dyn-opt}).  DIDO uses different inner loops to support its three performance goals of robustness, speed and accuracy as stated in Secs.~\ref{sec:intro} and \ref{sec:overview}.  Consequently, the three major algorithmic components of DIDO's suite of algorithms comprise (see Fig.~\ref{fig:DIDOMain3-repeat}):
%----------------------------------
\begin{enumerate}
\item A stabilization component: The task of this component of the algorithm is to drive an ``arbitrary'' tuple $(X, U, V, \Lambda, \Omega)^N_m$ to an ``acceptable'' starting point for the acceleration component.
\item An acceleration component: The task of this suite of algorithms is to rapidly guide the sequence of \eqref{eq:dido-c-sequence} to a capture zone of the accurate component.
 \item An accurate component: The task of this component is to generate a solution that satisfies the precision requested by the user.
\end{enumerate}
% ------------------------
%
%======================================================================================
   \begin{figure}[h!]
      \centering
      \framebox{\parbox{0.9\columnwidth}{
      \centering
      {\includegraphics[angle=0, width = 0.9\columnwidth]{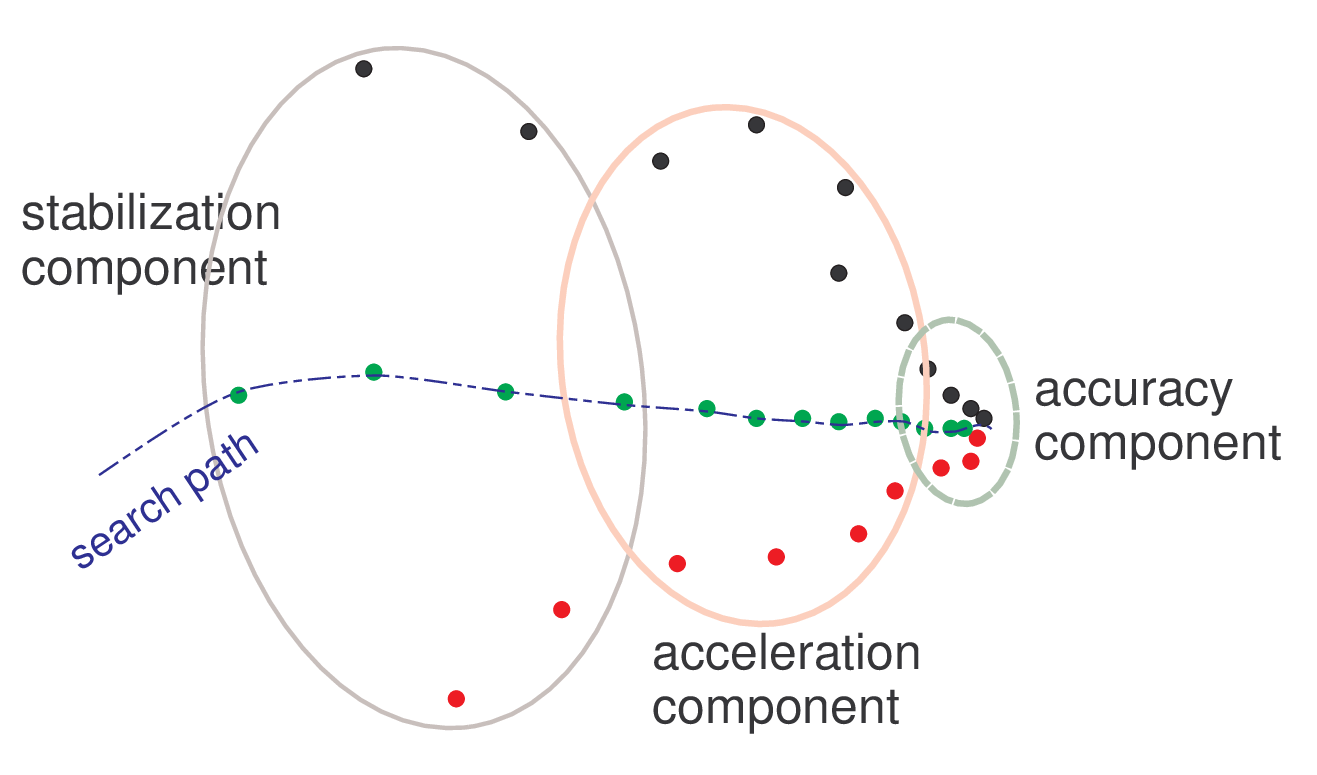}}
      \caption{\textsf{A schematic for the concept of the three main algorithmic components of DIDO. Same as Fig.~\ref{fig:DIDOMain3}, repeated here for ready reference.}}
    \label{fig:DIDOMain3-repeat}
    }}
   \end{figure}
%==========================================================================================
%
%
The three-component concept is based on the overarching theory introduced in \cite{rossJCAM-1}.  According to this theory, an algorithm may be viewed as a discretization of a controlled dynamical system where the control variable is a search vector or its rate of change\cite{ross-accel}.  In the case of static optimization, the latter case generates accelerated optimization methods that include Polyak's heavy ball method, Nesterov's accelerated gradient method and the conjugate gradient algorithm\cite{ross-accel}. In applying this theory for dynamic optimization, the DIDO-sequence given by \eqref{eq:dido-c-sequence} may be viewed as a discretization of a continuous primal-dual dynamical system.  The convergence theory proposed in \cite{rossJCAM-1,ross-accel} rests on the guidability (i.e., a weaker form of controllability) of this primal-dual system.  \textbf{\emph{New minimum principles}}\cite{rossJCAM-1,ross-accel} based on Lyapunov functions provides sufficient conditions for guidability.  The specifics of a Lyapunov-stable algorithm rely on the local convexity of the continuous search trajectory. Because convexity cannot be guaranteed along the search path (see Fig.~\ref{fig:DIDOMain3-repeat}), a suite of algorithms is needed to support different but simple objectives.  \emph{\textbf{DIDO is based on the idea that the three components of stability, acceleration and accuracy achieve the requirements of a true-fast algorithm}}. DIDO's main algorithm ties these three components together by monitoring their progress and triggering automatic switching between components based on various criteria as described in the subsections that follow.

\subsection{Details of the Three Components of DIDO}
Each of the three components of DIDO contain several subcomponents. Many of these subcomponents are described in detail across several published articles.  Consequently, we lean heavily on these articles in describing the details of DIDO's algorithms.

\subsubsection{\textbf{The Stabilization Component}}
There are three main elements to the stabilization component of DIDO:
\begin{enumerate}
\item[a)] The Guess-Free Algorithm;
\item[b)] Backtracking and Restart; and
\item[c)] Moore-Smith Sequencing.
\end{enumerate}
The details of these three elements are as follows:
\subsubsection*{a) The Guess-Free Algorithm}
In its nominal mode, DIDO is guess-free; i.e., a user does not provide any guess.  Thus, one of the elements of the stabilization component is to generate the first point of the sequence denoted in \eqref{eq:dido-c-sequence}.  Because the guess-free element itself requires a starting point, this phase of the algorithm runs through a sequence of prioritized starting points and each starting point is used as an input for an elastic programming algorithm described in \cite{guess-free} and \cite{spec-alg}. The second priority point is executed only if the first priority point terminates to an infeasible ``solution.''  If all priority points terminate to infeasibility, then DIDO will execute a final attempt at feasibility by minimizing a weighted sum of infeasibilities.  If this phase is not successful, then DIDO will present this solution to the user but tag it as infeasible.  If, at any point, a feasible solution is found, the guess-free element of the algorithm terminates with a success flag. This point is denoted as $(X, U, V, \Lambda, \Omega)^{N_0}_{m_0}$, the first point of the DIDO sequence. This is the Superior To Intuition$^{\text{TM}}$ concept first announced in 2008; see \cite{guess-free} and \cite{spec-alg}.

\subsubsection*{b) Backtracking and Restarts}
The second element of the stabilization component is triggered if the monitor algorithm detects instability (i.e., divergence) in the acceleration component. The task of this element of the stabilization component is to use a ``backtracked'' point\footnote{Because algorithmic backtracking in dynamic optimization can be computationally expensive, DIDO saves a small number of prior iterates that can be instantaneously recalled for rapid restarts.} to generate an alternative starting point for the acceleration component by reusing some of the elements of the guess-free algorithm.  The production of this alternative point is performed through a priority suite of subalgorithms that comprise restarts and elastic programming\cite{guess-free,spec-alg}.

\subsubsection*{c) Moore-Smith Sequencing}
The third element of the stabilization component is triggered if the monitor algorithm detects rapid destabilization in the acceleration component that cannot be handled by backtracking alone.  In this situation, a backtracked point is used as a starting point to generate sequential perturbations in Sobolev space\cite{PSReview-ARC-2012,Kang_2008_convergence} by generating continuous-time candidate solutions via interpolation. The perturbed discretized vector is used as starting point to generate a new sequence.  If this new sequence continues to exhibit instability, the entire process is repeated until stability is achieved or the number of Sobolev perturbations exceeds a predetermined ceiling. In the latter case, DIDO returns the last iterate as an infeasible solution.  The perturbations are also stopped if the projected variations exceed the bounds specified in the search space; see \eqref{eq:search-space} in Sec.~\ref{sec:DIDO-i/o}.  Because this sequence generation does not fit within the standard Arzel\`{a}-Ascoli theorem, we denote it as Moore-Smith\cite{yosida,McShane:moore-smith}.

\subsubsection{\textbf{The Acceleration Component}}
The nominal phase of the acceleration component is to rapidly generate the DIDO sequence (Cf.~\eqref{eq:dido-c-sequence}) subsequent to the first point handed to it by the stabilization component.  DIDO's Hamiltonian algorithm (see Definition~\ref{def:DIDO-c-inner}) produces this sequence. This algorithm uses Lemma~\ref{lemma:ross-funda-1} and a tentative value of $\delta^{\lambda, D, N_k}$ for each $N_k, k = 1, 2, \ldots$; see Remark~\ref{remark:lemma-14c}. The values of $\delta^{\lambda, D, N_k}$ for $k > 2$ are revised based on the solution generated by the previous sequence and the predicted rate of convergence.  The predicted rate of convergence is given by formulas derived in \cite{spec-alg,Kang_2008_convergence}.  Although the equations derived in \cite{spec-alg,Kang_2008_convergence} are for Lagrange PS methods, it is not too difficult to show that they are valid for Birkhoff PS methods as well because of the equivalence conditions derived in \cite{birk-koeppen}.

The theory for the fast spectral algorithm presented in \cite{spec-alg} is founded on an applications of the Arzel\`{a}-Ascoli theorem\cite{yosida} to PS discretization\cite{Kang_2008_convergence}.  Because the guarantees provided by Arzel\`{a}-Ascoli theorem are only for a subsequence, an implementation of this powerful result is not straightforward\cite{PSReview-ARC-2012,Kang_2008_convergence}.  Hence, we augment this result using the theory of optimization presented in \cite{rossJCAM-1}.  According to this theory, the DIDO sequence denoted in \eqref{eq:dido-c-sequence} may be viewed as a discretization of a controllable continuous ``search trajectory;''  see Fig.~\ref{fig:DIDOMain3-repeat}.  The dynamics of this search trajectory is given by a primal-dual controllable system that can be guided to a candidate optimal solution. Minimum principles based on Lyapunov functions provides sufficient conditions for a successful guidance algorithm\cite{rossJCAM-1,ross-accel}.  The Lyapunov function forms a merit function; and, under appropriate technical conditions\cite{rossJCAM-1,ross-accel}, the search trajectory can be designed to drive an arbitrary point to a stable point.  The stable point is defined in terms of the necessary conditions for optimality. For optimal control, these necessary conditions are given by Pontryagin's Principle\cite{ross-book}.  A first-principles' application of these ideas to analyze a direct shooting method method is described in \cite{ross:direct-shooting}.

A guarantee of convergence is achieved in terms of local convexity of a Riemannian metric space that can be defined using a regularized Hessian;\footnote{Because the ``exact'' Hessian may not be positive definite, a ``correction'' can better ensure progress towards the Lyapunov-stable point\cite{rossJCAM-1}.} see \cite{ross-accel} for theoretical details.  Because any point in an iteration can get trapped in a nonconvex region, subalgorithms are needed to detect and escape the trap regions. Detection of trap points can be made by monitoring the decrease in the Lyapunov function. Escape from suspected trap points is made by a call to the stabilization component of DIDO described previously.

An acceleration of the DIDO sequence, denoted by \eqref{eq:dido-c-sequence}, is achieved through a combination of PS discretizations\cite{lncis,PSReview-ARC-2012, birk-koeppen,birk-TN,arb-grid, cheby:+cmt}; see also Fig.~\ref{fig:PStypes}.   As shown by Theorem~\ref{theorem:Birk=tunnel}, the family of Birkhoff PS discretizations offers a special capability of tunneling through the cotangent barrier.  Consequently, an accelerated spectral sequencing is achieved by coordinating the progress of the outer loop with the inner-loop sequence as stipulated in Definition \ref{def:DIDO-c-inner}. Perhaps the most important aspect of a Birkhoff PS discretization is that it offers a ``flat'' condition number of the resulting linear system\cite{birk-koeppen, wang}; i.e., an $\mathcal{O}(1)$ growth with respect to $N$. In contrast, the condition number of a Lagrange PS discretization, be it Gauss, Gauss-Radau or Gauss-Lobatto, grows as $\mathcal{O}(N^2)$ no matter the choice of the orthogonal polynomial\cite{birk-koeppen,birk-TN}.  Because condition number affects both accuracy and computational speed, its combination with its cotangent tunneling capability makes a Birkhoff PS discretization a doubly potent method to generate accelerated algorithms.

%Because this search trajectory is in a Riemanian metric space that requires local convexity, a subalgorithm is needed to drive the search trajectory back to an Arzel\`{a}-Ascoli sequence whenever a subsequence deviates from a prediction. The prediction is based on the rate of decrease of the Lyapunov merit function. Thus a predictor-corrector inner-loop algorithm forms the basis of generating a fast and stable primal-dual Arzel\`{a}-Ascoli sequence.

\subsubsection{\textbf{The Accuracy Component}}
Technically, this component of DIDO is the simplest; however, its computer implementation requires a significant amount of skill in practical and algorithmic matters.  This is because the accuracy component of DIDO's algorithmic suite is handed a near-final solution; hence, its ``only'' technical task is to deliver the final solution.  
This point suggests that an implementation of the accuracy component can be a simple quasi-Newton algorithm (or its accelerated version as presented in \cite{ross-accel}) for solving $(P^{\lambda, D, N_k})$ for the last few values of $N_k$.  While the preceding arguments are technically correct, the main implementation challenge here is in ensuring that the momentum gained from the acceleration component is not lost in achieving the targeted accuracy requested by the user.  In achieving its stated goals, the iterations of this algorithm are largely focused on generating the targeted accuracy of satisfying the Pontryagin necessary conditions.  As noted earlier in Sec.~\ref{sec:overview}, this component of DIDO tightens the tolerances on the Hamiltonian minimization condition that may not have been met at the termination of the accelerated algorithm.  This procedure is done sequentially to maintain dual feasibility that is consistent with the value of $N = N_f$.  The mathematical details of this process are described in \cite{PSReview-ARC-2012,arb-grid-aas,arb-grid,cheby:+cmt}.

\subsection{An Overview of the Nominal, True-Fast, Spectral Algorithmic Flow}
Because the nominal mode of DIDO is guess-free; i.e., a user does not provide a guess, the first element of the stabilization component is invoked to generate the first element of the DIDO sequence given by $(X, U, V, \Lambda, \Omega)^{N_0}_{m_0}$; see \eqref{eq:dido-c-sequence}.  The acceleration component of DIDO rapidly iterates this point to completion or near-completion denoted by $(X, U, V, \Lambda, \Omega)^{N_a}_{m_a}$.  If this point satisfies all the termination criterion, the main algorithm terminates; otherwise, $(X, U, V, \Lambda, \Omega)^{N_a}_{m_a}$  is handed off to the accurate component which iterates it to $(X, U, V, \Lambda, \Omega)^{N_f}_{m_f}$.  Because $(V, \Omega)^{N_f}_{m_f}$ are internal variables, DIDO discards these virtual primal and dual variables.  The remainder of the variables that support a solution to $P^{\lambda, N_f}$ are passed to the user through the use of  \textsf{primal} and \textsf{dual} structures described in Sec.~\ref{sec:DIDO-i/o}.  For  $P^{\lambda, N_f}$ these variables are passed according to,
\begin{subequations}\label{eq:primal+dual4PN}
\begin{align}
&&&\textsf{primal}.\textbf{states} & = && X^{N_f} && \\
&&&\textsf{primal}.\textbf{controls} & = && U^{N_f} && \\
&&&\textsf{dual}.\textbf{dynamics} & = && \Lambda^{N_f} &&
\end{align}
\end{subequations}
\textbf{\emph{It is clear from the entirety of the preceding discussions that most of DIDO's algorithms are completely agnostic to the specifics of a PS discretization.}} Consequently, they can be adapted to almost any method of discretization!

\end{document}